\documentclass[reqno,11pt]{amsart}

\usepackage[a4paper,left=23mm,right=23mm,top=30mm,bottom=30mm,marginpar=25mm]{geometry}
\usepackage{amsmath}
\usepackage{amssymb}
\usepackage{amsthm}
\usepackage{amscd}
\usepackage{mathtools}
\usepackage{url}
\usepackage[ansinew]{inputenc}
\usepackage{color}
\usepackage[final]{graphicx}
\usepackage{esint} 
\usepackage{bbm}
\usepackage{subfigure}
\usepackage{tikz}
\usepackage{cite}
\usepackage{enumitem}
\usepackage[hidelinks]{hyperref}

\renewcommand{\epsilon}{\varepsilon}

\numberwithin{equation}{section}

\newtheoremstyle{thmlemcorr}{10pt}{10pt}{\itshape}{}{\bfseries}{.}{10pt}{{\thmname{#1}\thmnumber{ #2}\thmnote{ (#3)}}}
\newtheoremstyle{thmlemcorr*}{10pt}{10pt}{\itshape}{}{\bfseries}{.}\newline{{\thmname{#1}\thmnumber{ #2}\thmnote{ (#3)}}}
\newtheoremstyle{defi}{10pt}{10pt}{\itshape}{}{\bfseries}{.}{10pt}{{\thmname{#1}\thmnumber{ #2}\thmnote{ (#3)}}}
\newtheoremstyle{remexample}{10pt}{10pt}{}{}{\bfseries}{.}{10pt}{{\thmname{#1}\thmnumber{ #2}\thmnote{ (#3)}}}
\newtheoremstyle{ass}{10pt}{10pt}{}{}{\bfseries}{.}{10pt}{{\thmname{#1}\thmnumber{ A#2}\thmnote{ (#3)}}}

\theoremstyle{thmlemcorr}
\newtheorem{theorem}{Theorem}
\numberwithin{theorem}{section}
\newtheorem{lemma}[theorem]{Lemma}
\newtheorem{corollary}[theorem]{Corollary}
\newtheorem{proposition}[theorem]{Proposition}

\theoremstyle{thmlemcorr*}
\newtheorem{theorem*}{Theorem}
\newtheorem{lemma*}[theorem]{Lemma}
\newtheorem{corollary*}[theorem]{Corollary}
\newtheorem{proposition*}[theorem]{Proposition}
\newtheorem{problem*}[theorem]{Problem}
\newtheorem{conjecture*}[theorem]{Conjecture}

\theoremstyle{defi}
\newtheorem{definition}[theorem]{Definition}

\theoremstyle{remexample}
\newtheorem{remarkx}[theorem]{Remark}
\newenvironment{remark}
  {\pushQED{\qed}\remarkx}
  {\popQED\endremarkx}
\newtheorem{example}[theorem]{Example}

\theoremstyle{ass}

\newcommand{\Ccal}{\mathcal{C}}

\newcommand{\Ecal}{\mathcal{E}}
\newcommand{\Fcal}{\mathcal{F}}
\newcommand{\Gcal}{\mathcal{G}}

\newcommand{\Kcal}{\mathcal{K}}
\newcommand{\Lcal}{\mathcal{L}}

\newcommand{\Ncal}{\mathcal{N}}

\newcommand{\Pcal}{\mathcal{P}}

\newcommand{\Qcal}{\mathcal{Q}}

\newcommand{\Scal}{\mathcal{S}}

\newcommand{\Cfrak}{\mathfrak{C}}

\newcommand{\Pfrak}{\mathfrak{P}}

\DeclareMathOperator*{\argmin}{arg\,min}

\DeclareMathOperator{\diverg}{div}

\DeclareMathOperator{\dist}{dist}

\DeclareMathOperator{\supp}{supp}

\newcommand{\normlr}[1]{\left\|#1\right\|}

\newcommand{\normb}[1]{\bigl\|#1\bigr\|}

\newcommand{\abslr}[1]{\left|#1\right|}

\newcommand{\absb}[1]{\bigl|#1\bigr|}

\newcommand{\N}{\mathbb{N}}
\newcommand{\R}{\mathbb{R}}

\newcommand{\C}{\mathbb{C}}

\newcommand{\weakly}{\rightharpoonup}

\newcommand{\eps}{\epsilon}

\newcommand{\Nzero}{N^{s, p, \delta}(\Omega)}
\newcommand{\Nzerom}{N^{s, p, \delta}(\Omega;\R^m)}

\newcommand{\NzeroRn}{N^{s, p, \delta}(\R^n)}

\newcommand{\pro}{\pi_\delta^s}

    \newcommand{\vertiii}[1]{{\left\vert\kern-0.25ex\left\vert\kern-0.25ex\left\vert #1 
    \right\vert\kern-0.25ex\right\vert\kern-0.25ex\right\vert}}

\newcommand{\Ccg}{\Cfrak^{s,p, \d}}
\newcommand{\Pg}{\Pfrak^{s, p, \d}}

\DeclareMathOperator{\Div}{div}

\def\XXint#1#2#3{{\setbox0=\hbox{$#1{#2#3}{\int}$}
\vcenter{\hbox{$#2#3$}}\kern-.5\wd0}}

\DeclareMathOperator{\Id}{Id}

\DeclarePairedDelimiter\abs{\lvert}{\rvert}
\DeclarePairedDelimiter{\norm}{\lVert}{\rVert}
\DeclarePairedDelimiter{\inner}{\langle}{\rangle}

\newcommand{\Rn}{\R^{n}}

\renewcommand{\phi}{\varphi}

\DeclareMathOperator{\sign}{sign}

\newcommand{\starto}{\stackrel{*}{\rightharpoonup}}
\newcommand{\weakto}{\rightharpoonup}

\def\XXint#1#2#3{{\setbox0=\hbox{$#1{#2#3}{\int}$}
     \vcenter{\hbox{$#2#3$}}\kern-.5\wd0}}

\newcommand{\Rmn}{\mathbb{R}^{m \times n}}

\newcommand{\Lip}{\mathrm{Lip}}
\makeatletter
\g@addto@macro\bfseries{\boldmath}
\makeatother

\newcommand{\Hspd}{H^{s,p,\delta}}

\usepackage{esint}

\usepackage{appendix}

\DeclareMathOperator{\diver}{div}

\renewcommand{\O}{\Omega}

\renewcommand{\d}{\delta}

\def\XXint#1#2#3{{\setbox0=\hbox{$#1{#2#3}{\int}$}
		\vcenter{\hbox{$#2#3$}}\kern-.5\wd0}}

\definecolor{colora}{RGB}{21,128,66}
\definecolor{colorb}{RGB}{21,66,128}
\definecolor{colorc}{RGB}{128,21,66}

\usepackage{lmodern}
\usepackage{graphicx}
\usepackage{wrapfig}

\newcommand{\abss}[1]{\langle #1 \rangle}
\newcommand{\Hquo}{\widetilde{H}^{s,p,\d}(\Omega)}
\newcommand{\Wquo}{\widetilde{W}^{1,p}(\Omega)}

\title[Non-constant functions with zero nonlocal gradient]{Non-constant functions with zero nonlocal gradient and their role in nonlocal Neumann-type problems}

\author{Carolin Kreisbeck}
\address{Mathematisch-Geographische Fakult\"at, Katholische Universit\"at Eichst\"att-Ingolstadt, Osten\-stra{\ss}e 28, 85072 Eichst\"att, Germany}
\email{carolin.kreisbeck@ku.de}

\author{Hidde Sch\"{o}nberger}
\address{Mathematisch-Geographische Fakult\"at, Katholische Universit\"at Eichst\"att-Ingolstadt, Osten\-stra{\ss}e 28, 85072 Eichst\"att, Germany}
\email{hidde.schoenberger@ku.de}

\begin{document}
\maketitle

\thispagestyle{empty}
\begin{abstract}  
This work revolves around properties and applications of functions whose nonlocal gradient, or more precisely, finite-horizon fractional gradient, vanishes. Surprisingly, in contrast to the classical local theory, we show that this class forms an infinite-dimensional vector space. Our main result characterizes the functions with zero nonlocal gradient in terms of two simple features, namely, their values in a layer around the boundary and their average. 
The proof exploits recent progress in the solution theory of boundary-value problems with pseudo-differential operators. 
We complement these findings with a discussion of the regularity properties of such functions and give illustrative examples. Regarding applications, we provide several useful technical tools for working with nonlocal Sobolev spaces when the common complementary-value conditions are dropped. Among these, are new nonlocal Poincar\'e inequalities and compactness statements, which are obtained after factoring out functions with vanishing nonlocal gradient.  
Following a variational approach, we exploit the previous findings to study a class of nonlocal partial differential equations subject to natural boundary conditions, in particular, nonlocal Neumann-type problems. 
Our analysis includes a proof of well-posedness and a rigorous link with their classical local counterparts via $\Gamma$-convergence as the fractional parameter tends to $1$.

\vspace{8pt}

\vspace{8pt}

\color{black}
\vspace{8pt}

\noindent\textsc{Date:} \today.
\end{abstract}

\section{Introduction}

It is well-known that differentiable functions with zero gradient  
are exactly the constant functions, that is, for any open and connected set $\Omega\subset \R^n$ and $u\in C^1(\Omega)$ it holds that 
\begin{align}\label{zerograd_class}
\text{$\nabla u=0$ in $\Omega$\qquad  if and only if \qquad $u$ is constant on $\Omega$,}
\end{align}
and the same is true (almost everywhere) for Sobolev functions with weak gradients. One may wonder if this fundamental observation carries over when considering fractional and nonlocal derivatives instead of classical derivatives. As intriguingly basic as the question may sound, a universal answer is not easily available and depends on the specific setting, as we will demonstrate in the following. \smallskip

In fractional and nonlocal calculus, the study of gradient operators has attained increasing attention in recent years, see e.g.,~\cite{Shieh1, Comi1, Sil20, MeS15, BCM20, BCM23, KrS22, EGM22}. The nonlocal gradient of a function $u:\R^n\to \R$ is of the form 
	 	\begin{align} \label{eq: definition of nonlocal gradient}
			\Gcal_\rho u(x)= \int_{\R^n}  \frac{u(x)-u(y)}{|x-y|}\frac{x-y}{|x-y|}\rho(x-y)\, dy
		\end{align}
with a suitable kernel function $\rho$, whenever the integral is defined.
	
Especially the Riesz fractional gradient, that is, $D^s=\Gcal_{\rho^s}$ with $\rho^s \propto \abs{\,\cdot\,}^{-(n+s-1)}$ for $s\in (0,1)$, has been popular. Not only does $D^s$ have unique natural invariance and homogeneity properties~\cite{Sil20}, it also lends itself to a distributional approach towards fractional function spaces~\cite{Comi1}; in fact, the function spaces associated with $D^s$ in analogy to the classical Sobolev spaces coincide with the Bessel potential spaces $H^{s, p}(\R^n)$, as observed in~\cite{Shieh1}. The combination of these features make $D^s$ a good choice of fractional derivative, both from the applied point of view and 
in the context of variational theories and PDEs.   
	
In contrast, a compactly supported, radial kernel $\rho$ in~\eqref{eq: definition of nonlocal gradient} reduces the nonlocal interaction between all of $\R^n$ to points within an interaction range $\delta>0$, commonly referred to as horizon.  
By cutting-off the Riesz potential kernel, one obtains the finite-horizon fractional gradient defined as
\begin{align*}
D_\d^s = \Gcal_{\rho_\d^s} \quad \text{with} \quad \rho_\d^s \propto \frac{w_\d}{|\cdot|^{n+s-1}}
\end{align*}
where $w_\d:\R^n\to [0, 1]$ is a radial cut-off function supported in a ball of radius $\d$; for further properties, we refer to Section~\ref{sec:sobolevspaces}. These gradients $D_\d^s$, which we simply call nonlocal gradients in the following, are the key objects in this paper. They were first considered in~\cite{BCM23} by Bellido, Cueto \& Mora Corral (see also~\cite{CKS23}), motivated by applications in materials science. Since the nonlocal gradients inherit the desirable properties from the Riesz fractional gradients, while being suitable for variational problems on bounded domains, they have become the core of a newly proposed model for nonlocal elasticity. 

On a more technical note, we remark that in order to properly determine $D_\d^s u$ in $\Omega$, the function $u$ needs to be defined in the set $\Omega_\d=\{x\in \R^n: \dist(x, \Omega)<\d\}$ enlarged by the horizon variable, and in particular, in the collar $\Gamma_\d:=\Omega_\d\setminus \overline{\Omega}$ of thickness $\delta>0$ around $\Omega$, cf.~Figure~\ref{fig:setandcollars}.  
	The values of $u$ in $\Gamma_\d$ can be viewed as nonlocal boundary values. 
One defines the space $H^{s,p,\d}(\Omega)$ as the functions in $L^p(\Omega_\d)$ with $D_\d^s u\in L^p(\Omega)$; see  Section~\ref{sec:sobolevspaces} for more details.
As a powerful tool, we wish to point out the translation mechanism established in~\cite{CKS23, BCM23} (cf.~Section~\ref{subsec:translation}). It relates the nonlocal and local setting in the sense that nonlocal gradients can be expressed as classical ones and vice versa, allowing for statements to be carried over; in formulas, we have
\begin{align}\label{translation_intro}
D_\d^s =\nabla\circ (Q_\d^s\ast \,\cdot\,) \qquad\text{and} \qquad \nabla = \Pcal_\d^s \circ D_\d^s,
\end{align}
where $Q_\delta^s$ is an integrable, compactly supported kernel function and $\Pcal_\d^s$ corresponds to the inverse of the convolution with $Q_\delta^s$. In comparison with the analogous results for the Riesz fractional gradient \cite{KrS22, Shieh1}, the operator $\Pcal_\d^s$ replaces the fractional Laplacian of order $\frac{1-s}{2}$. 
  \smallskip

Let us now return to and specify the question raised earlier:
\[
\text{Is \eqref{zerograd_class} still true when $\nabla$ is replaced with $\Gcal_\rho$?}
\]
In the case of the Riesz fractional gradient $\Gcal_\rho=D^s$ on $H^{s, p}(\R^n)$ for $p \in (1,\infty)$, 
the answer is affirmative as a consequence of the fractional Poincar\'{e}-type inequalites~\cite[Theorem~1.8, 1.10 and 1.11]{Shieh1}; in fact, the functions with vanishing Riesz fractional gradient must even be zero due to their integrability properties. The same is true when $\Gcal_\rho=D^s_\d$ is considered for functions in the complementary-value space $H^{s, p, \d}_0(\Omega):=\{u\in H^{s,p, \d}(\R^n): u=0 \text{ a.e.~in $\Omega^c$}\}$. Here as well, a Poincar\'e inequality is available, see~\cite[Theorem~6.2]{BCM23}. If the complementary-value is dropped, however, and one considers nonlocal gradients $\Gcal_\rho=D_\delta^s$ on 
$H^{s,p, \d}(\Omega)$ with bounded $\Omega$, the picture changes substantially. 

This paper revolves around the class of functions with zero nonlocal gradient
 \begin{align*}
\Nzero:=\{h \in H^{s,p,\d}(\Omega)\,:\, D^s_\d h =0 \ \text{a.e.~in $\Omega$}\},
\end{align*}
which turns out to be non-trivial. Indeed, we show that there exist functions in $\Nzero$  that are non-constant in any open subset of $\Omega$ (Proposition~\ref{prop:nonconstant}) and establish that they are numerous in the sense that $\Nzero$ forms an infinite-dimensional space (Proposition~\ref{prop:Nzeroinfinitedim}). 

Knowing that the set $\Nzero$ consists of more than just constant functions stirs up interesting new issues for further investigation. We first give a characterization of all the elements of $\Nzero$, which provides a deeper understanding of its properties. 
With this at hand, we then highlight the role of the functions with zero nonlocal gradients in the theory of the spaces $H^{s, p, \d}(\Omega)$ and discuss applications in nonlocal differential inclusion problems and variational problems on $H^{s, p, \d}(\Omega)$ with Neumann-type boundary conditions. Here is a more detailed overview of our main new findings.\medskip

\color{black}
\textbf{Characterization of $\Nzero$.} While the set of functions in $W^{1,p}(\Omega)$ with zero gradient corresponds to the set of constant functions and can thus, be identified with $\R$ by taking mean values, 
the characterization of $\Nzero$ involves an additional feature due to boundary effects of the nonlocal interactions.
Roughly speaking, two ingredients are necessary to uniquely identify the elements of $\Nzero$, that is, an average or mean-value condition on $\Omega$ and boundary values in the collar region $\Gamma_\d :=\Omega_\d\setminus \overline{\Omega}$.
 
We start from the observation that $\Nzero$ consists of all functions $h\in L^p(\Omega_\delta)$ satisfying
\begin{align}\label{convolutionproblem_intro}
Q^s_\d*h=c \ \text{\ a.e.~in $\Omega$} \qquad \text{and} \qquad h=g \ \text{\ a.e.~in $\Gamma_\d$},
\end{align}
for a given $g \in L^p(\Gamma_\d)$ and $c \in \R$. This is a consequence of the translation mechanism~\eqref{translation_intro}. Hence, the problem reduces to finding the solutions of~\eqref{convolutionproblem_intro}. 
Since $\Pcal_\d^s$ from \eqref{translation_intro} is in fact a pseudo-differential operator, our proof strategy is to rewrite~\eqref{convolutionproblem_intro} equivalently as a pseudo-differential Dirichlet problem and to exploit the recent progress in their existence, uniqueness, and regularity theory. Precisely, the properties of $\Pcal_\d^s$ make it fit into the setting of the works by Grubb~\cite{Gru22} and by Abels \& Grubb~\cite{AbG23}. Given that the regularity results are sensitive to the relation between the fractional and integrability parameters $s$ and $p$, there are two qualitatively different regimes to distinguish. 

Our main characterization result (see Theorem~\ref{th:Ncalchar} and Proposition~\ref{th:Ncalchar_plarge}) states the following: 
\begin{itemize}
\item[$(i)$] If $p\in (1, \frac{2}{1-s})$ (including the case $p=2$), then $\Nzero$ consists of the unique solutions to~\eqref{convolutionproblem_intro}, which exist for every constant $c\in \R$ and given boundary values $g \in L^p(\Gamma_\d)$.\smallskip 
\item[$(ii)$] For $p \in [\frac{2}{1-s},\infty)$, only those (unique) solutions to~\eqref{convolutionproblem_intro} that lie also in $L^p(\Omega_\d)$ constitute $\Nzero$.
\end{itemize}
An alternative way of phrasing $(i)$ is to say that
\[
\text{$\Nzero$ is isomorphic to $\R\times L^p(\Gamma_\d)$,}
\]
with the isomorphism $\Nzero \ni h\mapsto (\int_\Omega Q_\d^s\ast  h\, dx, h|_{\Gamma_\d})$. This formalizes the statement that an average condition on $\Omega$ and the boundary values in a boundary layer of thickness $\d$ are the characteristics for any function with zero nonlocal gradient. Besides, we show that $h\mapsto (\int_\Omega h\, dx, h|_{\Gamma_\d})$ is a isomorphism between $\Nzero$ and $\R\times L^p(\Gamma_\d)$ as well, which indicates  that even a simple mean-value condition along with the values in $\Gamma_\d$ suffices to pin down the elements of $\Nzero$. Part $(ii)$ implies that the previous identifications with $\R\times L^p(\Gamma_\d)$ remain injective when $p \in [\frac{2}{1-s},\infty)$, however, surjectivity generally fails (Remark~\ref{rem:plargebcs}).

\medskip

\textbf{Technical tools in $H^{s,p, \d}(\Omega)$ modulo functions of zero nonlocal gradient.} 
The set $\Nzero$ can be used to develop new functional analytic tools for the spaces $H^{s,p,\d}(\O)$ without complementary-values. Unlike for $H^{s, p, \d}(\R^n)$ and $H^{s, p, \d}_0(\Omega)$,
however, analogues of the relevant tools and estimates in classical Sobolev spaces only hold in the quotient space $H^{s,p, \d}(\Omega)/\Nzero$, meaning modulo elements in $\Nzero$. With that in mind, we obtain the following:\smallskip

 \begin{itemize}
\item[\textit{(a)}] \textit{Refined translation mechanism for functions on bounded domains:}
We show in Theorem~\ref{th:connecquo} that the quotient space $H^{s, p, \d}(\Omega)/ \Nzero$ is isometrically isomorphic to $W^{1,p}(\Omega)$ modulo constants, meaning that one can identify $H^{s,p,\d}(\Omega)$ and $W^{1,p}(\Omega)$ up to functions with zero (nonlocal) gradient. The isomorphism turns nonlocal gradients into gradients.  \smallskip

\color{black}
\item[\textit{(b)}] \textit{Extension of functions from $H^{s,p, \d}(\Omega)$ to $H^{s,p, \d}(\R^n)$ up to $\Nzero$:} 
Even though an exact extension of functions in $H^{s, p, \d}(\Omega)$ to  $\R^n$ is generally not possible (cf. Example~\ref{ex}), there exists a bounded linear operator $\Ecal^s_\d:H^{s,p, \d}(\Omega) \to H^{s,p,\d}(\R^n)$, such that $\Ecal_\d^s u$ differs from $u$ in $\Omega_\d$ merely by a function with zero nonlocal gradient. \smallskip

\item[\textit{(c)}] \textit{Nonlocal Poincar\'e-type inequalities:}  As a major tool, we prove different nonlocal versions of Poincar\'e inequalities. If $p\in (1, \frac{1}{2-s})$, there exists a constant $C>0$ such that
\[
\norm{u}_{L^p(\Omega_\d)} \leq C\norm{D^s_\d u}_{L^p(\Omega;\R^n)} 
\]
for all $u \in H^{s, p, \d}(\Omega)$ satisfying $u=0$ in $\Gamma_\d$ and one of the averaging conditions $\int_\Omega u\, dx=0$ or $\int_\Omega Q_\d^s\ast u\, dx=0$. The same estimate holds for all $p \in (1,\infty)$ and $u\in H^{s, p, \d}(\Omega)$ whose metric projection onto $\Nzero$ vanishes. 

 \smallskip
\item[\textit{(d)}] \textit{$L^p$-compactness modulo $\Nzero$.} Based on $(c)$, we derive the following Rellich-Kondra\-chov-type compactness:  If $(u_j)_j\subset H^{s,p, \d}(\Omega)$ is a bounded sequence such that the metric projection of $u_j$ onto $\Nzero$ vanishes for all $j$, then $(u_j)_j$ is precompact in $L^p(\Omega_\d)$. 
\end{itemize}

We remark that for the complementary-value spaces $H_0^{s, p, \d}(\Omega)$ the analogues of $(a)$, $(c)$, and $(d)$ have recently been established in~\cite{BCM23, CKS23}. The approach there relies on Fourier techniques, given that the functions in $H_0^{s, p, \d}(\Omega)$ are defined on the whole of $\R^n$.
\medskip

\color{black}
\textbf{Variational problems on $\Nzero^\perp$ and nonlocal boundary-value problems.}
A significant application of the aforementioned tools are the existence theory and asymptotic analysis of nonlocal PDEs subject to Neumann-type boundary conditions. Precisely, we adopt a variational viewpoint and prove the existence of solutions to the problem
\begin{align}\label{functional_Fds_intro}
\text{Minimize } \quad \frac{1}{2}\int_{\Omega} \abs{D^s_\d u}^2\,dx-\int_{\Omega_\d}F u\,dx \quad \text{over $N^{s, 2, \d}(\Omega)^\perp\subset H^{s, 2,\d}(\Omega)$,} 
\end{align} 
where $\Omega\subset \R^n$ is a bounded Lipschitz domain, $F \in L^{2}(\Omega_\d)$, and $N^{s,2, \d}(\Omega)^\perp$ denotes the orthogonal complement of $N^{s, 2, \d}(\Omega)$; note that $N^{s,2,\d}(\Omega)$ plays the same role as the set of constant functions in the variational formulation of the Neumann problem with classical gradients. In fact, a remarkable aspect of our framework is that one can also  handle more general vector-valued nonlinear problems with $p \in (1,\infty)$ and energy densities that are either quasiconvex or polyconvex, see~Theorem~\ref{th:wellposed} and Remark~\ref{rem:wellposed}.

To draw the connection between \eqref{functional_Fds_intro} and PDEs with Neumann-type boundary conditions, one assumes the nonlocal compatibility condition
\[
\int_{\Omega_\d} Fh\,dx =0 \quad \text{for all $h \in N^{s,2,\d}(\O)$},
\]
under which the solutions to \eqref{functional_Fds_intro} weakly satisfy Euler-Lagrange equations with a nonlocal boundary operator $\Ncal^s_\d$ featured in the collar regions, see~\eqref{ELequation_split}. In fact, this boundary operator was recently introduced by Bellido, Cueto, Foss \& Radu~\cite{BCFR23}, where the authors derive, amongst others, a new integration by parts formula in the spirit of~\cite{DROV17}.

Our second main result regarding~\eqref{functional_Fds_intro} confirms the expectation that these problems localize as the fractional parameter $s$ tends to $1$, that is, they converge to their classical counterparts with usual gradients.
Working in the framework of variational convergence, we obtain that the $\Gamma$-limit of the functional in~\eqref{functional_Fds_intro} with respect to strong convergence in $L^2(\Omega_\d)$ is 
\begin{align*}
\displaystyle \frac{1}{2}\int_{\Omega}\abs{\nabla u}^2\,dx -\int_{\Omega}  F u \,dx \qquad\text{for $u\in W^{1,2}(\Omega)$ with $\int_\O u \,dx =0$,}
\end{align*}
see Theorem~\ref{th:local}; again, this also holds in the more general setting mentioned before.
In the case that $F$ satisfies the classical compatibility condition $\int_{\Omega} F\,dx = 0$, we obtain, in particular, that the minimizers of~\eqref{functional_Fds_intro} converge in $L^2(\Omega)$ as $s \uparrow 1$ to the unique mean-zero solution of the standard Neumann problem
\[
\begin{cases}
-\Delta u = F & \text{in $\Omega$},\\
\frac{\partial u}{\partial \nu} = 0 & \text{on $\partial \Omega$}.
\end{cases}
\]
Localisation via rigorous limit passages is a general theme in the study of fractional and nonlocal calculus, not least because they serve as important consistency checks for new problems and models; we refer e.g., to~\cite{BCM21, CKS23} for $s\uparrow 1$, and~\cite{MeD15, BMCP15, MeS15} for limits with vanishing horizon $\delta\to 0$. 

Let us close by pointing out some literature on Neumann problems in other fractional and nonlocal set-ups, involving the fractional Laplacian and more general integral and integro-differential operators, see~e.g.,~\cite{MuP19, BBC03, DeS15, AFNRO23, Gru22, DROV17} and also the references therein. One of the works we wish to highlight is~\cite{DROV17}, where Dipierro, Ros-Oton \& Valdinoci introduce a Neumann problem for the fractional Laplacian by a natural notion of normal nonlocal derivative. These results have been refined, expanded and generalized in various directions, e.g., in~\cite{AFNRO23, FoK23, DPLV22, DPLV23}. Closely related are also the recent results on nonlocal trace spaces~\cite{DTWY22, GrK23, GrH22, TiD17}. The distinguishing factor in our work, is the central role of a nonlocal gradient object, which enables us to handle a broad variety of nonlinearities.

\medskip

\textbf{Outline.} 
We have organized this paper as follows. After introducing notations and providing theoretical background and useful auxiliary results in Section~\ref{sec:prelim}, Section~\ref{sec:characterization_nonlocalgradfree} is centered around a solid understanding of the functions with vanishing finite-horizon fractional gradient. Our analysis includes proofs that non-constant functions with vanishing nonlocal gradient exist and that $\Nzero$ is an infinite-dimensional space, see Section~\ref{subsec:nonconstant}. The main theorems about the characterization of $\Nzero$ are stated and proven in Section~\ref{subsec:characterization}. We round off this section with a discussion of regularity properties of functions with zero nonlocal gradient and give illustrative examples in Section~\ref{subsec:regulex}.

The second part of the paper presents different implications and applications involving~$\Nzero$. In Section~\ref{sec:tools}, we establish the technical tools $(a)$-$(d)$ for working in the nonlocal function spaces $H^{s, p, \d}(\Omega)$. The previous findings are then used in Section~\ref{sec:diffinclusions} to contribute to the theory of nonlocal differential inclusions. We show that rigidity statements as well as existence results for approximate solutions can be carried over from the classical setting via the translation mechanism. Section~\ref{sec:neumann} features the new class of variational problems on $\Nzero^{\perp}$, which relates to nonlocal Neumann-type problems. 
A proof of well-posedness for these problems is contained in Section~\ref{sec:wellposed}, while Section~\ref{sec:localisation} establishes the rigorous link with the classical local problems through a localization result via $\Gamma$-convergence.

\section{Preliminaries}\label{sec:prelim}

In this section, we introduce the relevant notations and collect the necessary background on nonlocal gradients and function spaces along with some useful technical tools.

\subsection{Notation}  Unless specified otherwise in the following, we take $s\in (0,1)$, $p\in [1, \infty]$, and
 $\delta>0$.
\color{black}  The Euclidean norm of $x \in \R^n$ is denoted by $\abs{x}$ and
\begin{align*}
\abss{x} := \sqrt{1+\abs{x}^2}.
\end{align*}
We use the notation $l_A$ with $A\in \Rmn$ for the linear function $l_A(x)=Ax$ with $x\in \R^n$. 

Moreover, $E^c:=\R^n \setminus E$ is the complement of a set $E \subset \R^n$, $\overline{E}$ is its closure, and $|E|$ is its Lebesgue measure, provided $E$ is measurable. We use the notation $\mathbbm{1}_E$ for the indicator function of a set $E\subset \R^n$, i.e., $\mathbbm{1}_E(x)=1$ if $x\in E$ and $\mathbbm{1}_E(x)=0$ otherwise. Whenever convenient, we identify a function on a subset of $\R^n$ with its trivial extension by zero without explicit mention. If we wish to highlight the trivial extension, we use an extra indicator function, writing e.g.,~$\mathbbm{1}_E f:\R^n\to \R$ for the zero extension of $f:E\to\R$. The restriction of any $f:E\to \R$ to a subset $E'\subset E$ is denoted by $f|_{E'}$. 

By $B_\rho(x)=\{ y \in \R^n : \abs{x-y}<\rho\}$, we denote the ball centered at $x \in \R^n$ with radius $\rho>0$, and $\dist(x,E)$ is the distance between a point $x \in \R^n$ and a set $E\subset \R^n$.
For a domain $\Omega \subset \R^n$, i.e., open and connected set, we introduce its expansion and reduction by thickness $\delta$ as
\begin{align*}\label{Omegadelta}
\Omega_\delta:= \Omega+B_\d(0)=\{x \in \R^n\,:\, {\rm dist}(x,\Omega) <\delta\} \quad \text{and} \quad \Omega_{-\d}:=\{x \in \Omega \,:\, {\rm dist}(x,\partial \Omega) > \d\},
\end{align*}
where $\partial \Omega$ is the boundary of $\Omega$, and define
\begin{center}
$\Gamma_\d:=\Omega_\d\setminus \overline{\Omega}$ \quad and \quad $\Gamma_{-\d}:=\Omega\setminus \overline{\Omega_{-\d}}$
\end{center} as the inner and outer collars of $\Omega$, respectively. Further, let $\Gamma_{\pm \d}:=\Gamma_\d\cup \Gamma_{-\d}\cup \partial \Omega$ be the double layer around the boundary of $\Omega$. For an illustration of this geometric set-up, see Figure~\ref{fig:setandcollars}.
\color{black}

\begin{figure}[h]
\includegraphics[width=7.5cm]{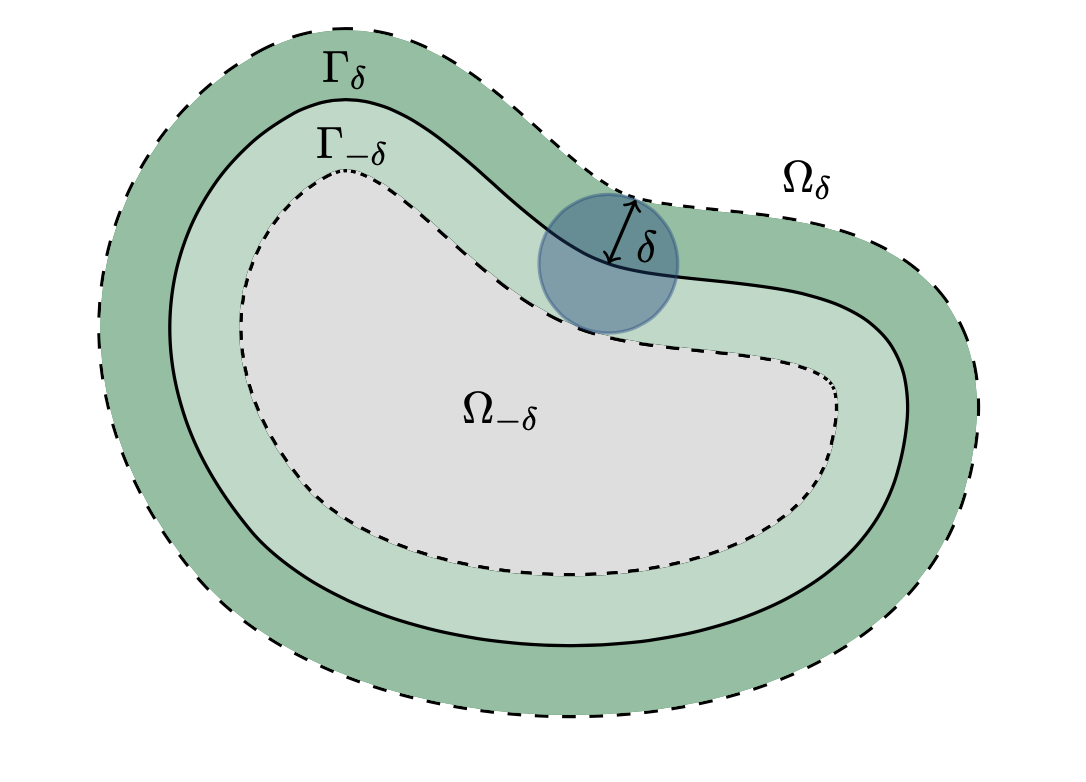}
\caption{Illustration of a set $\Omega\subset \R^n$ with its expansion $\Omega_\d$, the outer and inner collar regions $\Gamma_\d$ (green) and $\Gamma_{-\d}$ (light green), and the reduced set $\Omega_{-\d}$ (gray). }\label{fig:setandcollars}
\end{figure}

Let $U \subset \R^n$ be an open set. The notation  $C_c^\infty(U)$ stands for the space of smooth functions $U\mapsto \R$ with compact support, which will often be identified with their trivial extensions to $\R^n$ by zero, and $\Lip(\psi)$ is the Lipschitz constant of a function $\psi:\R^n \to \R$. Throughout the paper, we use the standard notation for Lebesgue- and Sobolev-spaces $L^p(U)$ and $W^{1,p}(U)$ with $p \in [1,\infty]$. For the inner product on $L^2(U)$, we write $\langle \cdot, \cdot\rangle_{L^2(U)}$. Notice that each of the function spaces defined above, as well as those introduced later, can be extended componentwise to vector-valued functions; the target set is then reflected in the notation, for example, $L^p(U;\R^m)$. Moreover, the restriction of a function space is denoted, for example, as $C^{\infty}(\R^n)|_{U}:=\{u |_{U} \,:\,u \in C^{\infty}(\R^n)\}$.

For an integrable function $f \in L^1(\R^n)$, the Fourier transform is defined as
\begin{equation*}
\widehat{f}(\xi):=\int_{\Rn} f(x) \, e^{-2\pi i x \cdot \xi} \, dx, \quad \xi \in \R^n.
\end{equation*}
It is well-known that the Fourier transform is an isomorphism from the Schwartz space $\mathcal{S}(\R^n;\C)$ onto itself, which can be extended to the spaces $L^2(\R^n;\C)$ and  the space of tempered distributions $\mathcal{S}'(\R^n;\C)$ by density and duality, respectively. \color{black} The inverse Fourier transform of $f$, denoted $f^\vee$, corresponds to $x \mapsto \widehat{f}(-x)$. For more background on Fourier analysis, see e.g.,~\cite{duon2000,Gra14a}. 

Lastly, $C$ denotes a generic constant, which may change from one estimate to the next without further mention. To indicate that a constant depends on specific quantities, they are added in brackets.

\subsection{Nonlocal gradients and Sobolev spaces}\label{sec:sobolevspaces}
Let us now introduce in detail the key objects in this paper, namely, a class of fractional gradients with finite horizon, and the associated nonlocal Sobolev spaces. Our presentation follows along the lines of~\cite{BCM23,CKS23} (see also~\cite{BCM23Eringen}), where we also refer to for more details.

The truncated Riesz fractional gradient, simply referred to as nonlocal gradient, and the corresponding divergence for smooth functions are defined as follows: 
	 For $s \in (0,1)$ and $\delta>0$, the nonlocal gradient of $\varphi\in C^{\infty}(\R^n)$ is \normalsize
	 	\begin{align} \label{eq: definition of nonlocal gradient}
			D_\delta^s \varphi(x)= \int_{ \R^n}  \frac{\varphi(x)-\varphi(y)}{|x-y|}\frac{x-y}{\abs{x-y}}\,\rho_{\d}^s(x-y)\, dy \quad \text{for $x \in \R^n$, }
		\end{align}
and the nonlocal divergence of $\psi \in C^{\infty} (\Rn;\Rn)$ is 
	\begin{align*}
			\diver_{\delta}^s \psi(x)=  \int_{ \R^n} \frac{\psi(x)-\psi(y)}{|x-y|}\cdot\frac{x-y}{\abs{x-y}}\,\rho_{\d}^s(x-y) \, dy \quad \text{for $x \in \R^n$;}
		\end{align*}
here, the kernel function $\rho_\d^s$ is given by
		\begin{align*}
		\rho_\d^s(z) = c_{n,s,\d} \frac{w_\d(z)}{|z|^{n+s-1}} \quad \text{ for $z\in \R^n\setminus \{0\}$, }
		\end{align*}
with $w_\d: \R^n\to [0,\infty)$ a non-negative cut-off function satisfying the hypotheses
\begin{enumerate}[label = (H\arabic*)]
	\item \label{itm:h1}$w_\d$ is radial, i.e., $w_\d=\overline{w}_\d(\abs{\,\cdot\,})$ with a function $\overline{w}_\d:\R \to [0,\infty)$;\\[-0.3cm]
	\item \label{itm:h2}$w_\delta$ is smooth and compactly supported in $B_\delta(0)$, i.e., $w_\d \in C_c^\infty(B_\d(0))$;\\[-0.3cm]
	\item \label{itm:h3} $w_\delta$ is normalized around the origin, i.e., $w_\d =1$ on $B_{\mu\delta}(0)$ for some $\mu\in (0,1)$;\\[-0.3cm]
	\item \label{itm:h4}$w_\delta$ is radially decreasing, i.e., $w_\d (z) \geq w_\d (\tilde z)$ if $|z| \leq |\tilde z|$,
\end{enumerate}
and the scaling constant $c_{n,s,\d}>0$ is such that
	\begin{align}\label{scalingfactor}
		c_{n,s,\d}\int_{B_\d(0)} \frac{w_\d(z)}{\abs{z}^{n+s-1}}\,dz=n.
	\end{align}	
	
\begin{remark}\label{rem:dsdconv}
a) Note that the scaling factor $c_{n, s,\d}$, determined by~\eqref{scalingfactor}, is the same as in \cite{BCM23Eringen} in order to ensure that the nonlocal derivative of any linear map $l_A$ with $A \in \Rmn$ is equal to $A$, see~\cite[Proposition~4.1]{BCM23Eringen}. This choice is slightly different from the scaling in \cite{CKS23}, but provides no substantial issues for the application of these results, as discussed in Remark~\ref{rem:scaling} below.\smallskip

b) An alternative way of expressing the nonlocal gradient in~\eqref{eq: definition of nonlocal gradient} 
is as a principle value integral. In view of the radial symmetry of $w_\delta$ from \ref{itm:h1}, one has for $x\in \R^n$ that
\begin{align}\label{nonlocalgrad_alter}
 D_\delta^s \varphi(x) = \text{p.v.}~\int_{B(x,r)^c} \varphi(y) d_\delta^s(x - y) \, dy  :=  \lim_{r \downarrow 0}\int_{B(x,r)^c} \varphi(y) d_\delta^s(x - y) \, dy 
\end{align}
 with $d_\delta^s(z) =-c_{n, s, \d} \frac{zw_\delta(z)}{|z|^{n+s+1}}$ for $z\in \R^n\setminus\{0\}$. This shows, in particular, that $D^s_\d \phi(x)$ can be written as the convolution of $d_\d^s$ with $\varphi$, when $x\notin \supp(\phi)$.
\end{remark}

The above definitions can be extended to locally integrable functions via a distributional approach. 
Indeed, for $\varphi \in C^{\infty}_c (\Rn)$ and $\psi \in C_c^{\infty} (\Rn; \Rn)$, the integration by parts formula 
\begin{align*}
\int_{\R^n} D_\delta^s \varphi \cdot \psi \, dx =& - \int_{\R^n} \varphi \diver_\delta^s \psi \, dx
\end{align*}
holds. 
Based on this, we then define $v \in L^1_{\rm loc}(\O;\R^n)$ as the weak nonlocal gradient of $u \in L^1_{\rm loc}(\Omega_\d)$, written as $v = D^s_\d u$, if
\begin{equation}\label{eq:intbyparts}
\int_{\O}v \cdot \psi \,dx = - \int_{\O_\d} u \Div^s_\d \psi\,dx \qquad\text{for all $\psi \in C_c^{\infty}(\O;\R^n)$};
\end{equation}
the weak nonlocal divergence is defined analogously.
 In parallel to classical Sobolev spaces, one can introduce nonlocal Sobolev spaces as follows.

\begin{definition}[Nonlocal Sobolev spaces]\label{def:Hspd}
Let $s\in (0,1)$, $\delta>0$, $p \in [1,\infty]$, and let $\O \subset \R^n$ be open. The nonlocal Sobolev space $\Hspd(\O)$ is defined as
\[
\Hspd(\O):=\{u \in L^p(\O_\d) \,:\, D^s_\d u \in L^p(\O;\R^n)\},
\]
endowed with the norm
\begin{displaymath}\label{normHspdelta}
\lVert u\rVert_{H^{s,p,\d}(\O)} = \Bigl( \left\| u \right\|_{L^p (\O_{\d})}^p + \| D_\delta^s u \|_{L^p (\O;\R^{n})}^p \Bigr)^{\frac{1}{p}}.
\end{displaymath}
\end{definition}

These spaces can be equivalently defined via density if $\Omega$ is a bounded Lipschitz domain or $\Omega=\R^n$, see~\cite[Theorem~1]{CKS23}. A more detailed study of these spaces, including results such as Leibniz rules, Poincar\'{e} inequalities and compact embeddings can be found in~\cite{BCM23, CKS23}.  \color{black}\smallskip

When working with functions on the full space $\R^n$, we will often exploit the connection between the nonlocal Sobolev spaces of Definition~\ref{def:Hspd} and the well-known Bessel potential spaces, which are defined for any $t\in \R$ and $p\in (1, \infty)$ as
\begin{equation}\label{eq:besselpotspaces}
H^{t,p}(\R^n) =\bigl\{ u \in \Scal'(\R^n) \,:\, \bigl(\abss{\,\cdot\, }^t \widehat{u}\bigr)^{\vee} \in L^p(\R^n)\bigr\},
\end{equation}
with the norm $\norm{u}_{H^{t, p}(\R^n)}=\norm{ \bigl(\abss{\,\cdot\, }^t \widehat{u}\bigr)^{\vee}}_{L^p(\R^n)}$ and the notation $H^{t}:=H^{t,2}$;
for more on the theory of Bessel potential spaces, see e.g.,~\cite[Chapter~1.3.1]{Gra14b} or \cite{Tri83}. 
Indeed, it holds for all $p\in (1, \infty)$ and $s\in (0,1)$ that
\begin{align*}
H^{s, p, \d}(\R^n) = H^{s, p}(\R^n),
\end{align*} 
with equivalent norms. This follows from the observation that $H^{s,p}(\R^n)$ coincides with the space of functions in $L^p(\R^n)$ with a weak Riesz fractional gradient in $L^p(\R^n;\R^n)$~(cf.~\cite[Theorem~1.7]{Shieh1} together with the density result in \cite[Theorem~A.1]{Comi3}), along with the fact that the latter is again the same as $H^{s,p, \d}(\R^n)$ due to~\cite[Lemma~5]{CKS23}.

We mention here some additional properties of the Bessel potential spaces that we need. First of all, for each $t > 0$, there is a $f_{t} \in L^1(\R^n)$ with $\norm{f_t}_{L^1(\R^n)}=1$ and $\widehat{f}_t=\abss{\cdot}^{-t}$ ($f_t$ is a rescaled version of the Bessel potential function, see~\cite[Chapter~1.2.2]{Gra14b}). Therefore, for any $ t_1 < t_2 $ we find with Young's convolution inequality
\begin{align}
\begin{split}\label{eq:increasingnorm}
\norm{u}_{H^{t_1,p}(\R^n)} &= \norm{ \bigl(\abss{\,\cdot\, }^{t_1} \widehat{u}\bigr)^{\vee}}_{L^p(\R^n)}=\norm{f_{t_2-t_1}* \bigl(\abss{\,\cdot\, }^{t_2} \widehat{u}\bigr)^{\vee}}_{L^p(\R^n)} \\
&\leq \norm{ \bigl(\abss{\,\cdot\, }^{t_2} \widehat{u}\bigr)^{\vee}}_{L^p(\R^n)}=\norm{u}_{H^{t_2,p}(\R^n)}.
\end{split}
\end{align}
Secondly, if $(u_j)_j \subset H^{t,p}(\R^n)$ is a bounded sequence and $t>0$, then we find that
\[
u_j = f_{t} * \bigl(\abss{\,\cdot\, }^{t} \widehat{u}_j\bigr)^{\vee},
\]
which is the convolution of an $L^1$-function with a bounded sequence in $L^p$, and hence, precompact in $L^p_{\mathrm{loc}}(\R^n)$ by the Fr\'{e}chet-Kolmogorov criterion (see~e.g.,~\cite[Corollary~4.28]{Brezis}). As such, $H^{t,p}(\R^n)$ is compactly embedded into $L^p_{\mathrm{loc}}(\R^n)$. In fact, when $t p > n$, we find that $f_{t} \in L^{p'}(\R^n)$ with $p'$ the dual exponent of $p$ (cf.~\cite[Theorem~1.3.5\,(c)]{Gra14b}), so that one can even deduce the compact embedding of $H^{t,p}(\R^n)$ into $C_{\mathrm{loc}}(\R^n)$ due to the Arzel\`{a}-Ascoli theorem.

In the following, we also use the complementary value space of $ H^{t,p}(\R^n)$ for $t \geq 0$, which consists of functions with zero values outside of an open set $V
\subset \R^n$ and is denoted by
\[
H^{t,p}_0(V)=\{u \in H^{t,p}(\R^n)\,:\, u=0 \ \text{a.e.~in $V^c$}\}.
\]

\subsection{Translation operators}\label{subsec:translation}
In this section, we present a method that will be frequently used, and further refined (see Section~\ref{sec:app1}) in this paper, namely, a translation procedure that allows switching between nonlocal gradients and classical gradients. The following auxiliary results are mainly taken from~\cite{CKS23}; related statements about the Riesz fractional gradient have been established in~\cite{KrS22}.  

\color{black}
Our starting point is the following finite-horizon analogue of the Riesz potential from \cite{BCM23}, defined by
\begin{align}\label{eq:Qdeltas}
Q_\d^s : \Rn \setminus \{0\} \to \R, \quad Q^s_\d(x)=c_{n,s,\d}\int_{\abs{x}}^\d \frac{\overline{w}_{\d}(r)}{r^{n+s}}\,dr.
\end{align} 
It holds that $Q^s_\d$ is integrable with compact support in $B_\delta (0)$ and, a simple calculation yields that, due to the choice of scaling, 
\begin{align}\label{int=1}
\norm{Q^s_\d}_{L^1(\R^n)}=1.
\end{align}

\begin{remark}\label{rem:scaling}
a) With the scaling constant $c_{n,s}$ used in \cite{CKS23}, one obtains instead of~\eqref{int=1} that $[0,1) \ni s \mapsto \norm{Q^s_\d}_{L^1(\R^n)}$ is continuous with
$
\lim_{s \to 1} \norm{Q^s_\d}_{L^1(\R^n)} =1,
$
see~\cite[Lemma~6]{CKS23}. This shows that the two different scaling regimes are comparable uniformly in $s$. \smallskip

b) The Fourier transform of $\widehat{Q}^s_\d$ is a smooth, positive and radial function. Moreover, the difference between $\widehat{Q}^s_\d$ and $\xi \mapsto \abs{2\pi \xi}^{-(1-s)}$ is a Schwartz function for $\abs{\xi} \geq 1$, see~\cite{BCM23} and \cite[Remark~2 and Lemma~11]{CKS23}.
\end{remark}

An essential observation about the kernel function $Q_\d^s$ regards its relation with the nonlocal gradient $D_\d^s$, that is,
\[
D^s_\d \phi = \nabla (Q^s_\d * \phi) = Q^s_\d * \nabla \phi \qquad  \text{for any $\phi \in C_c^{\infty}(\R^n)$.}
\]
This identity can be extended to the Sobolev spaces in a weak sense, as shown in~\cite[Theorem~2\,$(i)$]{CKS23}.

\color{black}
\begin{lemma}[From nonlocal to local gradients] 
\label{le:translation1}
Let $s\in (0,1)$, $\delta>0$, $p \in [1,\infty]$, and $\O \subset \R^n$ be open. 
Then, the linear map $\Qcal_\delta^s: H^{s, p, \delta}(\Omega)\to W^{1,p}(\Omega), \ u\mapsto Q_\delta^s\ast u$ is bounded (uniformly with respect to $s$) with 
\begin{align*}
(\nabla \circ \Qcal_\d^s) u =\nabla (\Qcal_\delta^s u) = D^s_\delta u \quad\text{for every $u\in \Hspd(\Omega)$}.
\end{align*}
\end{lemma}

The convolution with the kernel $Q^s_\d$ enables us to pass from the nonlocal Sobolev space to the classical one. To go back, we consider the operator from~\cite{CKS23}, given by  
\begin{align}\label{Pdeltas}
\Pcal^s_\d:\Scal(\R^n) \mapsto \Scal(\R^n), \quad \widehat{\Pcal^s_\d\phi} = \frac{\widehat{\phi}}{\widehat{Q^s_\d}}.
\end{align}
It is proven in~\cite[Theorem~2\,$(ii)$]{CKS23} that this operator 
can be extended to the Sobolev space as the inverse of convolution with $Q^s_\d$. 
\color{black}
\begin{lemma}[From local to nonlocal gradients]\label{le:translation2}
Let $s\in (0,1)$, $\delta>0$,  $p \in [1,\infty]$. Then, $\Pcal^s_\d$ in~\eqref{Pdeltas} can be extended to a isomorphism between $W^{1,p}(\R^n)$ and $H^{s,p,\d}(\R^n)$ with satisfies $(\Pcal^s_\d)^{-1}=\Qcal_\delta^s$. In particular, 
\begin{align}\label{translation_formula}
(D_\d^s \circ \Pcal_\d^s) v= D^s_\d (\Pcal^s_\d v) = \nabla v \quad \text{for every $v \in W^{1,p}(\R^n)$. }
\end{align}
\end{lemma}

We mention a useful alternative representation of $\Pcal_\d^s$ involving the kernel function of the nonlocal fundamental theorem of calculus in~\cite[Theorem~4.5]{BCM23}. It holds that
\begin{equation}\label{eq:vsdrepr}
\Pcal^s_\d \phi(x) = \int_{\R^n}V^s_\d(x-y) \cdot \nabla \phi(y)\,dy=: (V^s_\d * \nabla \phi) (x)  \quad \text{for $\phi \in C_c^{\infty}(\R^n)$},
\end{equation}
where $V^s_\d \in C^{\infty}(\R^n\setminus\{0\};\R^n)\cap L^1_{\rm loc}(\R^n;\R^n)$ is a vector-radial function, i.e., 
$V^s_\d(x)=xf_\delta^s(\abs{x})$ for $x\in \R^n\setminus\{0\}$ with $f_\delta^s:(0,\infty)\to \R$, 
  cf.~\cite[Remark~4\,d)]{CKS23} as well as~\cite[Theorem~5.9]{BCM23} for more properties of $V_\d^s$.

When $p \in (1,\infty)$, 
one can deduce some more general properties for $\Pcal^s_\d$ using Fourier methods. 
 To this end, we recall that by Remark~\ref{rem:scaling}\,b), there are $R^s_\d, S_\d^s \in \Scal(\R^n)$ such that
\begin{equation}\label{eq:decayR}
\widehat{Q}^s_\d(\xi) = \frac{1}{\abs{2\pi\xi}^{1-s}} + R^s_\d(\xi) \quad \text{and}\quad  \frac{1}{\widehat{Q}^s_\d(\xi)}=\abs{2\pi\xi}^{1-s} + S^s_\d(\xi)  \qquad\text{for $\abs{\xi}\geq 1$}.
\end{equation}
As a consequence of the Mihlin-H\"ormander theorem (e.g., \cite[Theorem~6.2.7]{Gra14a}), using the smoothness and positivity of $\widehat{Q}^s_\d$ locally and the decay of $R^s_\d, S^s_\d$ for large $\xi$ to obtain the desired estimates similarly to \cite[Lemma~8]{CKS23}, it follows that both 
\begin{align}\label{comparison:QBessel}
\abss{\cdot}^{1-s}\widehat{Q}^s_\d \quad \text{and} \quad
 \frac{1}{\abss{\cdot}^{1-s}\widehat{Q}^s_\d} 
\end{align}
are $L^p$-multipliers.
We finally infer from this observation (along with the definition of Bessel potential spaces in \eqref{eq:besselpotspaces}) that
 for $t \geq 1-s$ and $p\in (1, \infty)$,
\begin{equation}\label{eq:psdhom}
\Pcal^s_\d:H^{t,p}(\R^n) \to H^{t-(1-s),p}(\R^n),
\end{equation}
is a isomorphism with inverse  $(\Pcal^s_\d)^{-1}=\Qcal^s_\d$. 

\color{black} Moreover, since the decay of $R^s_\d$ is uniform in $s$, a similar argument as in \cite[Lemma~8]{CKS23} shows that the operator norm of $\Pcal^s_\d$ is uniformly bounded in $s\in (0,1)$; in particular, using \eqref{eq:increasingnorm} and $H^{0,p}=L^p$, there is a $C>0$ independent of $s$ such that
\begin{equation}\label{eq:uniformpbound}
\norm{\Pcal^s_\d v}_{L^p(\R^n)} \leq \norm{\Pcal^s_\d v}_{H^{s,p}(\R^n)} \leq C \norm{v}_{W^{1,p}(\R^n)} \quad \text{for all $v \in W^{1,p}(\R^n)$ and $s \in (0,1)$.}
\end{equation} 

\subsection{Pseudo-differential operators and Dirichlet problems}\label{sec:pseudo} 
The recent existence and uniqueness theory from \cite{Gru22} together with the regularity results in \cite{AbG23} for boundary-value problems involving pseudo-differential operators play an important role for our work. We collect here the statements that we will need, while keeping the presentation accessible, and refer to e.g.,~\cite[Section~2.2]{Gru22} for precise definitions and properties of pseudo-differential operators.

For a suitable pseudo-differential operator $\Pcal$, an open subset $V \subset \R^n$, and a function $g:V \to \R$, the associated Dirichlet problem reads
\begin{align}\label{pseudoprob_general}
\begin{cases}
\Pcal w=g & \text{ on $V$}\\
 w=0& \text{ in $V^c$. }
\end{cases}
\end{align}
By combining the results of  \cite{Gru22, AbG23}, we obtain into the following statement tailored to our needs.

\color{black}
\begin{theorem}[Existence and uniqueness for pseudo-differential Dirichlet problems]
\label{th:grubb}
Let $p\in (1, \infty)$, $a \in (0,1/2)$, $V \subset \R^n$ be an open and bounded set with $C^{1,1}$-boundary, and let $\Pcal$ be a strongly elliptic, even, classical pseudo-differential operator of order $2a$ satisfying
\begin{equation}\label{eq:garding}
\inner{\Pcal\phi,\phi}_{L^2(\R^n)} \geq C\norm{\phi}^2_{H^{a}(\R^n)} \quad \text{for all $\phi \in C_c^{\infty}(\R^n)$}
\end{equation}
with some constant $C>0$. Then, 
there exists for every $g \in L^p(V)$ a unique $w_g \in H^{a,p}_0(V)$ with 
\begin{center}
$\Pcal w_g=g$ \quad in $V$. 
\end{center}
If $p \in (1,\frac{1}{a})$, it even holds that $w_g \in H^{2a,p}_0(V)$ and there is a $c>0$ such that
\[
\norm{w_g}_{H^{2a,p}(\R^n)} \leq c\norm{g}_{L^p(V)} \quad \text{for all $g\in L^p(V)$. }
\]
\end{theorem}

\begin{proof} We define the operator $\Pcal_V$ with domain 
\[
\text{dom}(\Pcal_V):=\{w \in H^{a,p}_0(V)\,:\, (\Pcal w)|_V \in L^p(V)\}
\]
via restriction of $\Pcal$ to $V$, i.e., $\Pcal_Vw=(\Pcal w)|_V$ for $w\in \text{dom}(\Pcal_V)$. 
Due to \eqref{eq:garding}, we may apply \cite[Theorem~4.2]{Gru22} with $\beta=0$, and then also \cite[Theorem~4.16~2$^\circ$]{Gru22}, to deduce that $\Pcal_V:\text{dom}(\Pcal_V)\to L^p(V)$ is a bijection. \color{black} This shows the first part of the statement. 

For the case $p\in (1, \frac{1}{a})$, we note that in \cite{AbG23} (see also~\cite[Theorem~3.2]{Gru22}), the domain $\text{dom}(\Pcal_V)$ has been characterized as the so-called $a$-transmission space, which agrees with $H^{2a,p}_0(V)$ when $p \in (1,\frac{1}{a})$ (cf.~\cite[Eq.~(2.20)]{Gru22}). \color{black} Consequently, $\Pcal_V:H^{2a,p}_0(V) \to L^p(V)$ is a bijective bounded linear operator. In particular, it is invertible with bounded inverse, which implies the full statement.
\end{proof}

\color{black}
\begin{remark}\label{rem:grubbth}
a) We note that connectedness is not part of the definition of a domain in~\cite{Gru22, AbG23}, in contrast to our definition, and hence, Theorem~\ref{th:grubb} is valid for non-connected sets $V$ as well. Moreover, the regularity of the domain $V$ in Theorem~\ref{th:grubb} can even be reduced to $V \subset \R^n$ that have $C^{1,\tau}$-boundaries with $\tau\in (2a, 1)$, see~\cite{Gru22, AbG23}. Only for simplicity of the presentation, we work here with a stronger assumption. 
\smallskip

\color{black} b) Note that the range of $p$ such that $w_g$ lies in $H^{2a,p}_0(V)$ is sharp, which is due to the fact that the solutions to the pseudo-differential problem in~\eqref{pseudoprob_general} contain a factor of $\dist(\cdot,\partial V)^{a}$ near the boundary. To give a precise example, we can take any smooth, bounded domain $V$ and any function $w:\R^n \to \R$ that is smooth in $V$, equal to $\dist(\cdot,\partial V)^{a}$ near the boundary $\partial V$ and zero in $V^c$. It follows then by \cite[Theorem~4]{Gru15} that $(\Pcal w) |_{V} \in C^{\infty}(\R^n)|_{V}$, which implies, in particular, that $w\in H^{a, p}_0(V)$ is a solution to \eqref{pseudoprob_general} with a smooth right-hand side. However, we have that $w\dist(\cdot,\partial V)^{-2a}$ is equal to $\dist(\cdot,\partial V)^{-a}$ near the boundary of $V$, which is not in $L^p(V)$ for $p \geq \frac{1}{a}$, so that $w \not \in H^{2a,p}_0(V)$ in view of the Hardy-type inequality~\cite[Proposition~5.7]{Tri01}.
\end{remark}

The translation operator $\Pcal^s_\d$ from the previous section is in fact a pseudo-differential operator that fits exactly into the abstract framework of Theorem~\ref{th:grubb}, which is the content of the following lemma. This observation will be crucial for our characterization result of $\Nzero$ (cf.~Theorem~\ref{th:Ncalchar} and Lemma~\ref{le:solpdo2}).

\color{black}
\begin{lemma}[$\Pcal_\d^s$ as pseudo-differential operator]\label{lem:pdoproperties}
The operator $\Pcal^s_\d$ defined in~\eqref{Pdeltas} is a strongly elliptic, even classical  pseudo-differential operator of order $1-s$ and there is a $C>0$ such that
\begin{align*}
\inner{\Pcal^s_\d \phi,\phi}_{L^2(\R^n)} \geq C \norm{\phi}^2_{H^{\frac{1-s}{2}}(\R^n)} \quad \text{for all $\phi \in C_c^{\infty}(\R^n)$.}
\end{align*}
\end{lemma} 
\begin{proof} 
The properties can be deduced from the fact that the symbol of $\Pcal^{s}_\d$ is smooth, radial and positive, and for large frequencies only differs from the symbol of the fractional Laplacian $(-\Delta)^{\frac{1-s}{2}}$ up to a Schwartz function (see~\eqref{eq:decayR}). For the readers' convenience, we work out the details below, referring to~\cite[Section~2.2]{Gru22} for the precise definitions of the properties of pseudo-differential operators.

It is easy to check in light of \eqref{eq:decayR} that for any $\alpha \in \N_0^n$,
\[
\abslr{\partial^{\alpha} \bigl(1/\widehat{Q}^s_\d\bigr)} \leq C_{\alpha} \abss{\cdot}^{1-s-\abs{\alpha}},
\]
which means that $\Pcal^s_\d$ has order $1-s$. Defining $p_0:\R^n \to \R$ to be a smooth function with $p_0(\xi) = \abs{2\pi\xi}^{1-s}$ for $\abs{\xi} \geq 1$, we obtain again from \eqref{eq:decayR} that
\[
\abslr{\partial^{\alpha} \bigl(1/\widehat{Q}^s_\d-p_0\bigr)} \leq C_{\alpha}\abss{\cdot}^{1-s-\abs{\alpha}-J},
\]
for any $J \in \N_0$. This means that $\Pcal^s_\d$ is classical (where in the expansion $p_j=0$ for $j \geq 1$) and, since $p_0(-\xi)=p_0(\xi)$ for $\abs{\xi}\geq 1$, $\Pcal^s_\d$ is even. Finally, because $p_0(\xi) \geq C\abs{\xi}^{1-s}$ for $\abs{\xi}\geq 1$, the operator $\Pcal^s_\d$ is strongly elliptic, and since $1/\widehat{Q}^s_\d \geq C\abss{\cdot}^{1-s}$, we have by the Plancherel identity that  
\[
\inner{\Pcal^s_\d \phi,\phi}_{L^2(\R^n)} =  \inner{\widehat{\phi}/\widehat{Q}_\d^s,\widehat{\phi}}_{L^2(\R^n)} \geq C  \norm{\langle\cdot \rangle^{\frac{1-s}{2}}\widehat{\varphi}}_{L^2(\R^n)} = C\norm{\phi}_{H^{\frac{1-s}{2}}(\R^n)}
\]
for all $\phi \in C_c^{\infty}(\R^n)$, which finishes the proof.
\end{proof}

\section{Discussion and characterization of functions with zero nonlocal gradient }\label{sec:characterization_nonlocalgradfree}
This section revolves around the study of the functions in the nonlocal Sobolev space with vanishing finite-horizon fractional gradient. Our analysis  examines different facets of the set 
\begin{align*}
\Nzero=\{h \in H^{s,p,\d}(\Omega)\,:\, D^s_\d h =0 \ \text{a.e.~in $\Omega$}\}, 
\end{align*}
including the existence and construction of non-trivial functions, characterization results, a discussion of regularity properties, and illustrative examples. Throughout Sections~\ref{sec:characterization_nonlocalgradfree}-\ref{sec:neumann}, $\Omega$ is assumed to be a bounded Lipschitz domain, unless mentioned otherwise.

\subsection{Non-constant elements of $\Nzero$}\label{subsec:nonconstant}

In contrast to the functions with zero classical gradient, the set $\Nzero$ encompasses strictly more than constant functions. In fact, one can find functions in $\Nzero$ that are non-constant on any subset of $\Omega$, which is the content of the following proposition.

\begin{proposition}[Existence of non-constant functions in $\Nzero$]\label{prop:nonconstant}
Let 
$p\in [1, \infty]$. It holds for any open, non-empty $U\subset \Omega$ that 
\[
\Nzero \not \subset \{u \in H^{s, p, \delta}(\Omega)\,:\, u \ \text{is constant a.e.~on $U$}\}.
\]
\end{proposition}
\begin{proof} 
Suppose to the contrary that every $u\in \Nzero$ is constant a.e.~on $U$. The reasoning that will lead to the desired contradiction is organized in three steps.
\smallskip

\textit{Step~1: Representation of $V_\delta^s$ away from the origin.} We deduce from  the above assumption that the kernel $V_\delta^s$ in~\eqref{eq:vsdrepr} then has to satisfy
\begin{align}\label{contra_Vdeltas}
\text{$V^{s}_\d(x) = \frac{1}{\sigma_{n-1}}\frac{x}{\abs{x}^n}$ \quad  for all $x \in \overline{B_{\rho}(0)}^c$,}
\end{align} 
where $\rho=\mathrm{Diam}(\Omega)$ is the diameter of $\Omega$ and $\sigma_{n-1}$ denotes the surface area of the unit sphere in $\R^n$.\footnote{Note that the function $x\mapsto \frac{1}{\sigma_{n-1}} \frac{x}{|x|^n}$ for $x\in \R^n$ corresponds exactly to the kernel function appearing in the fundamental theorem of calculus for the classical gradient, see~\cite[Proposition~4.14]{Pon16}.}
To see~this, we split the argument in three sub-steps, showing first that $\Div V^s_\d$ is constant outside of $\overline{B_\rho(0)}$. Next, we exploit the radiality of $V_\delta^s$ (cf.~\ref{itm:h1}) and its boundedness away from the origin, which yields~a representation of $V_\delta^s$ on $\overline{B_\rho(0)}^c$ up to constants. The latter are then determined explicitly in the final step.\smallskip

\textit{Step 1a.} For every $\varphi \in C_c^{\infty}(\overline{\Omega}^c)$, we infer from~\eqref{translation_formula} that $D^s_\d (\Pcal_\delta^s \varphi)= \nabla \varphi = 0$ in $\Omega$, and hence, $\Pcal_\delta^s\varphi\in \Nzero$. By our initial assumption, $\Pcal_\delta^s \varphi$ is then constant on $U$, which, in view of  the identity $\Pcal_\delta^s \varphi = V_\delta^s\ast \nabla \varphi$ in~ \eqref{eq:vsdrepr}, is equivalent to
\[
\int_{\R^n} \bigl(V^s_\d(x-z)-V^s_\d(y-z)\bigr)\cdot \nabla \varphi(z)\,dz =0 \quad \text{for all $x,y \in U$}.
\]
Since this holds for any $\varphi \in C_c^{\infty}(\overline{\Omega}^c)$, the fundamental lemma of the calculus of variations in combination with integration by parts allows us to deduce
\begin{align}\label{eq87}
\Div V^s_\d(x-z) = \Div V^s_\d (y-z) \quad \text{for all $x,y \in U$ and all $z \in \overline{\Omega}^c$.}
\end{align}

Let us fix $x\in U$ and consider $w \in \R^n$ with $\abs{w} > \rho = \mathrm{Diam}(\Omega)$. It follows then that $x-w \in \overline{\Omega}^c$, and we obtain with $z=x-w \in \overline{ \Omega}^c$ that
\[
\Div V^s_\d(w) = \Div V^s_\d(w+y-x)
\]
for all $y \in U$. Taking $y \in B_{\epsilon}(x) \subset U$, with $\eps>0$ sufficiently small, yields
\[
\Div V^s_\d(w) = \Div V^{s}_\d(w') \quad \text{for all $w' \in B_{\epsilon}(w)$}.
\]
For $n>1$, this shows that the divergence of $V^s_\d$ is locally constant on the connected set $\overline{B_{\rho}(0)}^c$, and thus, constant outside of $\overline{B_{\rho}(0)}$; the case $n=1$ yields the same conclusion by also utilizing the vector-radiality of $V^s_\d$.  \smallskip

\textit{Step 1b.} 
 Recall that $V_\delta^s$ is smooth on $\R^n\setminus \{0\}$ and vector-radial, meaning that there is a smooth function $f_\delta^s:(0,\infty) \to \R$ with $V^s_\d(x) = x f_\delta^s(\abs{x})$. 
We can then rewrite the divergence of $V_\delta^s$ as
\[
\Div V^s_\d(x) = n f_\delta^s(\abs{x})+\abs{x}(f_\delta^s)'(\abs{x})
\] 
for $x\in \R^n \setminus\{0\}$.
Since this expression is constant on the complement of $\overline{B_{\rho}(0)}$ by Step~1, the auxiliary function $f_\delta^s$ satisfies for all $r> \rho$ the equation
\begin{align}\label{dgl}
n f_\delta^s(r) + r(f_\delta^s)'(r) = c
\end{align}
with some constant $c \in \R$. The family of solutions to the ordinary differential equation~\eqref{dgl} is given by $r\mapsto\frac{c}{n}+\frac{k}{r^n}$ with $k \in \R$. Consequently, there is a $k \in \R$ such that
\begin{align*}
V^s_\d(x) =c \frac{x}{n} + k\frac{x}{\abs{x}^n} \quad \text{for all $x \in \overline{B_{\rho}(0)}^c$}.
\end{align*}

\textit{Step~1c.} The boundedness of $V^{s}_\d$ on $\overline{B_{\rho}(0)}^c$ according to \cite[Theorem~5.9\,$b)$]{BCM23} implies $c=0$. 
To determine $k$, consider the compactly supported, integrable function 
\begin{align*}
 x\mapsto V^s_\d(x)-k\frac{x}{\abs{x}^n},  
\end{align*}
whose Fourier transform is continuous and can be calculated to be 
\begin{align}\label{aux11}
\xi\mapsto \frac{-i\xi }{2\pi |\xi|^2}\Bigl(\frac{1}{\widehat{Q}^s_\d(\xi)}-k\sigma_{n-1}\Bigr), 
\end{align}
see~\cite[Theorem~5.9]{BCM23}.  
 As the first factor in~\eqref{aux11} 
 has a pole at the origin, the second factor needs to vanish in $0$ because of continuity. Due to $\widehat{Q_\delta^s}(0)=\norm{Q_\delta^s}_{L^1(\R^n)}=1$, this eventually yields $k=\sigma_{n-1}^{-1}$, confirming~\eqref{contra_Vdeltas}.\medskip

\textit{Step~2: Entire extension of $\widehat{Q}_\delta^s$ is zero-free.} Let us introduce the auxiliary function $Z_\delta^s\in  C^{\infty}(\R^n\setminus\{0\};\R^n)\cap L^1(\R^n;\R^n)$ defined by
\begin{align*}
Z_\delta^s(x) = V^s_\d(x)-\frac{1}{\sigma_{n-1}}\frac{x}{\abs{x}^n}.
\end{align*}
As $Z_\delta^s$ has compact support owing to~Step~1, the Paley-Wiener theorem (see e.g.,~\cite[Theorem~2.3.21]{Gra14a}) implies that the Fourier transform $\widehat{Z}_\delta^s$ with
\begin{align}\label{FourierZ}
\widehat{Z}_\delta^s(\xi) = \frac{-i\xi }{2\pi |\xi|^2}\Bigl(\frac{1}{\widehat{Q}^s_\d(\xi)}-1\Bigr) \quad \text{for $\xi\in \R^n\setminus\{0\}$}. 
\end{align} 
 is real analytic and allows for a unique entire extension to a function $\C^n\to \C^n$. An analogous argument gives that the Fourier transform of the kernel function $Q_\delta^s$ (see~\eqref{eq:Qdeltas} and recall $\supp (Q_\delta^s)\subset B_\delta(0)$) is extendable (uniquely) to a holomorphic function $\C^n\to \C$. In the following, we write $\widehat{Z}_\delta^s$ and $\widehat{Q}_\delta^s$ for both the Fourier transforms of $Z_\delta^s$ and $Q_\delta^s$, as well as for their extended versions defined on $\C^n$. 

The goal in this step is to show that
\begin{align}\label{zerofree}
\text{$\widehat{Q}_\delta^s:\C^n\to \C$ is zero-free. }
\end{align}
Suppose for the sake of contradiction that this is not the case, and let $\zeta_0 \in \C^{n}\setminus\{0\}$ be a zero of $\widehat{Q}^s_\d$ with minimal norm $r:=\abs{\zeta_0}$; note that $\widehat{Q}_\d^s(0)=\norm{Q_\d^s}_{L^1(\R^n)}=1$. Applying the identity theorem of complex analysis, in each variable separately, to $\widehat{Z}^s_\d$ as in~\eqref{FourierZ} yields that
\begin{align*}
\widehat{Z}_\delta^s(\zeta) = \frac{-i\zeta }{2\pi |\zeta|^2}\Bigl(\frac{1}{\widehat{Q}^s_\d(\zeta)}-1\Bigr) \quad \text{for $\zeta\in \C^n \setminus \{0\}$ with $\abs{\zeta}<r$}.
\end{align*} 
We now find that $\lim_{r' \uparrow 1} |\widehat{Z}_\delta^s(r'\zeta_0)| = \infty$, which contradicts the continuity of the holomorphic extension of $\widehat{Z}_\delta^s$.  Thus,~\eqref{zerofree} is proven. \smallskip

\textit{Step~3: Section of $\widehat{Q}_\delta^s$ coincides with exponential of a polynomial.} Consider $q_\delta^s: \C\to \C,  z\mapsto \widehat{Q_\delta^s}(ze_1)$ with  $e_1=(1, 0, \ldots, 0)\in \C^n$. From $\norm{Q^s_\d}_{L^1(\R^n)}=1$ and $\supp(Q^s_\d) \subset B_\d(0)$, we conclude that
\[
\absb{q_\delta^s(z)}=\absb{\widehat{Q}^s_\d(ze_1)} \leq \int_{\R^n} \abs{Q^s_\d(x)}\abs{e^{-2\pi i x_1 z}}\,dx \leq e^{2\pi \delta \abs{z}}
\]
for all $z\in \C$, 
showing that $q_\delta^s$ is an entire function of order at most $1$. As a consequence of Step~2, this function is never zero, so that the Hadamard factorization theorem (see e.g.,~\cite[Corollary~XII.3.3]{Lan99}) yields that
\begin{align*}
q_\delta^s(z)=e^{az+b} \quad \text{ for all $z\in \C$ with $a,b \in \C$.}
\end{align*} However, this contradicts the fact that $q_\delta^s: z \mapsto \widehat{Q}^s_\d(ze_1)$ is non-constant and even (cf.~Section~\ref{subsec:translation}), as the section of the Fourier transform of the radial kernel $Q_\delta^s$,  which proves the statement.
\end{proof}

\color{black}

We point out that the previous result is not true when $\Omega=\R^n$. In fact, $\NzeroRn$ for $p<\infty$ contains only the zero function, which can be deduced with the help of the translation operators of Section~\ref{subsec:translation} as follows: Let $u\in \NzeroRn$, then $v:=\Qcal_\delta^s u\in W^{1,p}(\R^n)$ satisfies $\nabla v=D_\delta^s u = 0$, and hence, $v=0$, so that $u=\Pcal_\delta^s (\Qcal_\delta^s u) = \Pcal_\delta^s v = 0$. A similar argument for $p=\infty$, shows that $N^{s,\infty,\d}(\R^n)$ only consists of constant functions. Nevertheless, there are unbounded sets $\Omega\subset \R^n$ for which $\Nzero$ is non-trivial.

\begin{remark}[Generalization to unbounded sets] 
Proposition~\ref{prop:nonconstant} holds more generally for 
open sets $\Omega\subset \R^n$ such that $\overline{\Omega}^c$ contains the trace of an unbounded continuous curve $\gamma:[0,\infty)\to \R^n$.

The proof can easily be adjusted, with only a minor modification in Step~1a. After showing~\eqref{eq87} as above, let us fix $x\in U$ and consider $w=x-z\in \R^n$ for some  $z\in \gamma([0, \infty))$.
It follows then from \eqref{eq87} that $\diverg V_\delta^s(w) = \diverg V_\delta^s(w')$ for all $w'\in B_\eps(w)$ with $\eps>0$ such that $B_\eps(x)\subset U$. By applying this for all $z \in \gamma([0,\infty))$ and exploiting the radial symmetry of the divergence of $V_\delta^s$, we find that $\diverg V_\delta^s$ is constant in the complement of $\overline{B_\rho(0)}$ with $\rho:=\dist(x,\gamma([0,\infty)))$.
\end{remark}

 The following result confirms that there exist, in fact, a great many functions with vanishing nonlocal gradients, by showing that they form an infinite-dimensional space. 

\begin{proposition}[$\Nzero$ is infinite-dimensional]\label{prop:Nzeroinfinitedim}
Let $p\in [1, \infty]$, then $\Nzero$ is an infinite-dimensional closed subspace of $H^{s,p, \delta}(\Omega)$. 
\end{proposition}

\begin{proof} The proof idea relies on a semi-explicit construction procedure in three steps: Starting with $u \in H^{s,p,\d}(\Omega)$, there is a $v \in W^{1,p}(\Omega)$ with $\nabla v = D^s_\d u$ on $\Omega$ by Lemma~\ref{le:translation1}. Next, we extend $v$ in an arbitrary way to a compactly supported function in $W^{1,p}(\R^n)$. In view of Lemma~\ref{le:translation2}, it holds that $\Ecal_\d^s u:=\Pcal_\delta^s v \in H^{s,p,\d}(\R^n)$ satisfies $D^s_\d (\Ecal_\d^s u)=D^s_\d u$ on $\Omega$, or equivalently, 
\begin{align}\label{auxh}
h=(\Ecal_\delta^s u)|_{\Omega_\d}-u \in \Nzero.
\end{align}
Note that $\Ecal_\d^s u$ can be viewed as an extension operator of $u$ modulo a function with vanishing nonlocal gradient, see Section~\ref{subsec:extension} for more details.

Let $m\in \N$. With the aim of constructing $m$ linearly independent functions in $\Nzero$, we take $u_1,\ldots,u_{m} \in H^{s,p,\d}(\Omega)$ with $m\in \N$ such that no (non-trivial) linear combination of them can be extended to a function in $H^{s,p,\d}(\R^n)=H^{s,p}(\R^n)$. This can be achieved, for instance, if each function $u_j$ has a suitable singularity in  different places in the collar $\Gamma_\d$. \color{black}
 
To give more details, choose $x_1, \ldots, x_m$ as distinct points in $\Gamma_\d$ and let $\eps>0$ be such that the balls $B_{2\eps}(x_j)$ are pairwise disjoint and compactly contained in $\Gamma_\d$. Further, define 
\begin{align*}
u_j(x) = \mathbbm{1}_{B_{\eps}(x_j)}(x) u\Bigl(\frac{x-x_j}{\eps}\Bigr) \quad\text{for $x\in \Omega_\d$ and $j=1, \ldots, m$,}
\end{align*}
\color{black}
where $u\in L^p(\R^n)\setminus H^{s, p, \delta}(\R^n)$ with $\supp(u)$ compactly contained in $\overline{B_1(0)}$ (cf. Example~\ref{ex}). 
Note that any function in $L^p(\Omega_\delta)$ that is zero in a compact set containing $\Omega$ lies in $H^{s, p, \delta}(\Omega)$. Indeed, the nonlocal gradient is then given by a convolution and defines an $L^p$-function due to Young's convolution inequality, see~Remark~\ref{rem:dsdconv}\,b); the integration by parts formula in \eqref{eq:intbyparts} can be verified via Fubini's theorem. We conclude that $u_j\in H^{s,p,\delta}(\Omega)$ for all $j$. On the other hand, by construction no $u_j$ has an extension to $H^{s, p,\d}(\R^n)$, and thus, nor does any linear combination of $u_1, \ldots, u_m$.

According to~\eqref{auxh}, we now set $h_j=(\Ecal_{\d}^s u_j)|_{\Omega_\d} - u_j \in \Nzero$ for $j=1, \ldots, m$. If these functions $h_j$ were linearly dependent, then one could find a non-trivial linear combination of $u_1, \ldots, u_m$ that can be extended to $\R^n$ via the operator $\Ecal_\d^s$, which contradicts the assumption. Hence, $h_1, \ldots, h_m$ are linearly independent.
  Because $m\in \N$ is arbitrary, this shows that $\Nzero$ must be infinite-dimensional. 
\end{proof}

For the reader's convenience, we give here explicit examples of functions $u\in L^p(\R^n)\setminus H^{s, p, \delta}(\R^n)$ for all $1 \leq p \leq \infty$ with compact support in $\overline{B_1(0)}$, as they were used in the previous proof. 
\begin{example}\label{ex}
Let $p\in [1, \infty)$. Defining for some $0<\nu< \min\{\tfrac{n}{p},s\},$
\begin{align*}
u=\mathbbm{1}_{B_{1}(0)}\abs{\,\cdot\, }^{-\tfrac{n}{p}+\nu} \in L^p(\R^n),
\end{align*} 
gives a function with the desired properties if $\nu$ is sufficiently small. That $u \notin H^{s,p, \d}(\R^n)$, follows by contradiction with the estimate in \cite[Proposition~7.2]{BCM23}, once we have shown the existence of a constant $c>0$ such that
\begin{align}\label{eq:rough}
\norm{u-u(\cdot +h)}_{L^p(B_1(0))} 
 \geq c\, \abs{h}^{\nu} \qquad \text{for all $h \in  B_1(0)$.}
\end{align}
Indeed, we have that
\begin{align*}
\norm{u-u(\cdot +h)}_{L^p(B_1(0))} & \geq \norm{u}_{L^p(B_{|h|/2}(0))} - \norm{u(\,\cdot  +h)}_{L^p(B_{|h|/2}(0))}\\
&\geq \left(\int_{B_{\abs{h}/2}(0)} \abs{x}^{-n+\nu p}\,dx\right)^{1/p}- \left( \bigl|B_{|h|/2}(0)\bigr|\, (|h|/2)^{-n+\nu p}\right)^{1/p}\\ 
&\geq \left(\frac{c_1}{(\nu p)^{1/p}}-c_2\right) |h|^{\nu},
\end{align*}
with $c_1, c_2>0$, so that \eqref{eq:rough} follows for small $\nu$.

For larger $p$ there are also more elementary examples, such as the indicator function of a ball when $1<sp<\infty$, or any discontinuous function when $n<sp<\infty$ (see~e.g.,~\cite[Section~2.1]{BCM20}). For the case $p=\infty$, we can  take $u$ to be any discontinuous function with support in $B_1(0)$. Indeed, if it were true that $u \in H^{s,\infty,\d}(\R^n)$, then we also find that $u \in H^{s,q,\d}(\R^n)$ for all $q  \in [1,\infty)$ given its compact support. This would yield by \cite[Theorem~6.3]{BCM23}, that $u$ is H\"{o}lder continuous up to order $s$, which gives a contradiction.
\end{example}

\subsection{Characterization of $\Nzero$} \label{subsec:characterization}
Now that the presence of non-constant functions with zero nonlocal gradient is confirmed, the next task is to understand - and eventually, characterize - all functions in~$\Nzero$.

We start by observing that a function $u \in L^p(\Omega_\d)$ lies in $\Nzero$ if and only if $\Qcal_\d^su=Q^s_\d*u$ is constant (in $\Omega$). Indeed, if $Q_\d^s\ast u$ is constant, then
\[
\int_{\Omega_\d} u \Div^s_\d \phi\,dx = \int_{\R^n} u \,Q^s_\d*\Div\phi\,dx = \int_{\Omega} (Q^s_\d * u) \Div \phi\,dx=0 \quad \text{for all $\phi \in C_c^{\infty}(\Omega;\R^n)$,}
\]
which shows that $u \in H^{s,p,\d}(\Omega)$ with $D^s_\d u =0$. Conversely, if $u \in \Nzero$, it follows from Lemma~\ref{le:translation1} that $\Qcal^s_\d u \in W^{1,p}(\Omega)$ has zero gradient, and is thus, constant. The desired characterization of $\Nzero$ therefore comes down to identifying all solutions $h\in L^p(\Omega_\delta)$ for convolution equations of the form 
\begin{align*}
\text{$Q_\delta^s\ast h=c$ \ a.e.~in $\Omega$}
\end{align*} for any $c\in \R$.

Our strategy for characterizing $\Nzero$
starts at a suitable boundary-value problem involving the convolution operator $\Qcal_\delta^s$. Via inversion of $\Qcal_\d^s$ (based on the translation tools from Section~\ref{subsec:translation}), we then rewrite the latter  equivalently as a pseudo-differential equation featuring $\Pcal_\delta^s$ subject to Dirichlet boundary conditions, for which  a solution theory can be achieved. Let us make the mentioned equivalence precise.

\begin{lemma}[Equivalence between (C) and (P)]\label{le:equivpdo}
Let $p \in (1,\infty)$, $c \in \R$ and $g \in L^p(\Gamma_\d)$. Further, let $\Omega'\subset \R^n$ be any smooth and bounded set with $\Omega_{2\d} \subset \Omega'$ and $\Gamma':=\Omega' \setminus \overline{\Omega}$. Then, 
\begin{equation*} 
\text{\rm (C)}\ \begin{cases}
\Qcal^s_\d h = c& \text{a.e.~in $\Omega$},\\
h = g & \text{a.e.~in $\Gamma_\d$}, 
\end{cases}
\end{equation*}
has a solution $h\in L^p(\Omega_\delta)$ if and only if there exists a $w\in H^{1-s,p}(\R^n)$ solving 
\begin{equation*} 
\text{\rm (P)}\ \begin{cases}
\Pcal^s_\d w = \mathbbm{1}_{\Gamma_\d}g &\text{a.e.~in $\Gamma'$},\\
w = c &\text{a.e.~in $\overline{\Omega}$},\\
w = 0 &\text{a.e.~in $(\Omega')^c$}.\end{cases}
\end{equation*}
Specifically, the following holds:
\begin{itemize}
\item[$(i)$] If $h \in L^p(\Omega_\d)$ is a solution to {\rm (C)}, then $w:=\Qcal^s_\d(\mathbbm{1}_{\Omega_\d} h) \in H^{1-s,p}(\R^n)$ solves {\rm (P)}.
\smallskip
\item[$(ii)$] 
If $w \in H^{1-s,p}(\R^n)$ satisfies {\rm (P)}, then $h:=(\Pcal^s_\d w)|_{\Omega_\d}\in L^p(\Omega_\delta)$ is a solution to \text{\rm (C)}. 
\end{itemize}
\end{lemma}

\begin{proof}
The main ingredient of this proof is~\eqref{eq:psdhom} with $\beta=1-s$, according to which $\Pcal_\d^s$ is a isomorphism from $H^{1-s, p}(\R^n)$ to $L^p(\R^n)$ with inverse $\Qcal_\d^s$. 
 
The implication $(i)$ follows then immediately from the observation that 
\begin{align*}
\Pcal^s_\d w= \Pcal_\delta^s\Qcal_\delta^s (\mathbbm{1}_{\Omega_\d}h)= \mathbbm{1}_{\Omega_\d}h = \mathbbm{1}_{\Gamma_\delta}g \quad\text{ a.e.~in $\Omega^c$, }
\end{align*}
along with the property that $\supp(Q^s_\d) \subset B_{\d}(0)$, which implies $\supp(w)\subset \Omega_\delta +  B_\delta(0)\subset \Omega_{2\delta}\subset \Omega'$ as well as $w=\Qcal_\delta^s h =c$ a.e.~in $\overline{\Omega}$ given $\abs{\partial\Omega}=0$.

On the other hand, $(ii)$ holds since   
$\Qcal^s_\d h  = \Qcal_\delta^s (\Pcal_\delta^s w)= w =c$ a.e.~in $\Omega$ 
 and $h=\Pcal^s_\d w = g$ a.e.~in $\Gamma_\d $, again using that $\abs{\partial\Omega}=0$.
\end{proof}

Based on the previous lemma, we can express $\Nzero$ in terms of the solution sets of the boundary-value problems (C) and (P). To be precise, let
\begin{align}\label{eq:Nunion}
\Ccg(\Omega):=\bigcup_{c \in \R,  g\in L^p(\Gamma_\d)} \Ccg(c,g) \qquad\text{ and }\qquad \Pg(\Omega):=\bigcup_{c \in \R,  g\in L^p(\Gamma_\d)} \Pg(c,g),
\end{align} 
where $\Ccg(c,g)$ for $c \in \R$ and $g\in L^p(\Gamma_\d)$ denotes the set of all solutions in $L^p(\Omega_\d)$ to (C), and $\Pg(c,g)$ comprises the functions in $H^{1-s, p}(\R^n)$ solving (P). Then,
\begin{align}\label{Nzero0}
\Nzero= \Ccg(\Omega)=\Pcal_\delta^s\bigl(\Pg(\Omega)\bigr)|_{\Omega_\d}. 
\end{align}

We take~\eqref{Nzero0} as motivation to turn our attention to (P) and address the question of its solvability. It turns out that the recent existence and uniqueness results by Grubb \cite{Gru22} in combination with the regularity theory in~\cite{AbG23} by Abels \& Grubb for general pseudo-differential operators 
(see Theorem~\ref{th:grubb} for a version tailored to our setting)
 provides an answer. Even though the next lemma is a direct application of this abstract framework, we have included for illustration also a hands-on alternative proof in the case $p=2$. 

\begin{lemma}[Existence and uniqueness for (P)]\label{le:solpdo2}

Let  $p\in (1, \infty)$ and $\Omega$ be a bounded $C^{1,1}$-domain. Then, for every $c \in \R$ and $g\in L^p(\Gamma_\d)$, the problem \text{\rm (P)} admits a unique solution $w_{c,g}\in H^{\frac{1-s}{2}, p}(\R^n)$.
If $p\in (1, \frac{2}{1-s})$, then $w_{c,g}\in H^{1-s, p}(\R^n)$.
\end{lemma}

\begin{proof} 
We may assume without loss of generality that $c =0$, since for every $w \in \Pg(0,g-c)$ it holds that $w+\Qcal^s_\d(\mathbbm{1}_{\Omega_\d}c) \in \Pg(c,g)$. Hence, we can focus our attention to the pseudo-differential equation
\begin{equation}\label{eq:pzero}
\begin{cases}
\Pcal^s_\d w = \mathbbm{1}_{\Gamma_\d}g &\text{a.e.~in $\Gamma'$},\\
w = 0 &\text{a.e.~in $(\Gamma')^c$}. \end{cases}
\end{equation}
The statement now follows immediately by applying Theorem~\ref{th:grubb} for the pseudo-differential operator $\Pcal=\Pcal^s_\d$ and the set $V=\Gamma'$. Indeed, $\Pcal_\d^s$ satisfies all the required assumptions according to Lemma~\ref{lem:pdoproperties} and $\Gamma'=\Omega'\setminus \overline{\Omega}$ is a bounded open set with $C^{1,1}$-boundary, given that $\partial \Gamma' =\partial \Omega' \cup \partial \Omega$.\smallskip

Our alternative proof for $p=2$ relies on a familiar variational argument and exploits regularity results for the fractional Laplacian. 
 Let us consider the operator
\[
\Lcal_\delta^s:\Scal(\R^n) \to \Scal(\R^n), \quad  \widehat{\Lcal_\delta^s\varphi} = \frac{\widehat{\varphi}}{\sqrt{\widehat{Q^s_\d}}},
\]
which can be extended to a bounded linear operator $H^{t}(\R^n)\to H^{t - \frac{1-s}{2}}(\R^n)$ for any $t \in \R$. Then,
$(\Lcal_\delta^s)^2=\Lcal_\delta^s\circ \Lcal_\delta^s=\Pcal^s_\d$,
and we observe that $\norm{\Lcal_\delta^s\, \cdot}_{L^2(\R^n)}$ is a norm on $H^{\frac{1-s}{2}}(\R^n)$ that is equivalent to $\norm{\cdot}_{H^{\frac{1-s}{2}}(\R^n)}$ in view of~\eqref{comparison:QBessel}. 

As a consequence of the generalized Dirichlet principle, the 
 functional
\[
w \mapsto \norm{\Lcal_\delta^s w}_{L^2(\R^n)}^2-\int_{\R^n}\mathbbm{1}_{\Gamma_\d}{g}\,w\,dx
\]
over all functions $w \in H^{\frac{1-s}{2}}_0(\Gamma')$ has a unique minimizer $w_\ast$, which is also the unique solution to 
\begin{align}\label{eq:weakpdo}
\int_{\R^n} \Lcal_{\delta}^sw\;\Lcal_\delta^s \phi - \mathbbm{1}_{\Gamma_\d}{g}\phi\,dx =0 \quad \text{for all $\phi \in C_c^{\infty}(\Gamma')$}.
\end{align}
\color{black}
Since~\eqref{eq:weakpdo} is a weak formulation of~\eqref{eq:pzero} for $p=2$, the function $w_\ast\in H_0^{\frac{1-s}{2}}(\Gamma')$  is indeed the only candidate for the sought solution. 

It remains to prove that $w_\ast\in H^{1-s}(\R^n)$. To this end, we compare the operator $\Lcal_\delta^s$ with the fractional Laplacian $(-\Delta)^{\frac{1-s}{4}}: H^{\frac{1-s}{2}}(\R^n)\to L^2(\R^n)$, showing that they differ by a bounded linear operator on $L^2(\R^n)$. 
Indeed, $\Kcal_\delta^s:=\Lcal_\delta^s-(-\Delta)^{\frac{1-s}{4}}$ is an $L^2$-Fourier multiplier operator 
with multiplier 
\begin{align*}
m_\delta^s(\xi)=\frac{1}{\sqrt{\widehat{Q}^s_\d(\xi)}}-\abs{2\pi\xi}^{\frac{1-s}{2}} \qquad \text{ for $\xi\in \R^n$.}
\end{align*} The boundedness of $m_\delta^s$ follows from the smoothness and positivity of $\widehat{Q}_\delta^s$ together with the observation that for $|\xi|>1$,
\[
m_{\delta}^s(\xi) 
=\Biggl(\frac{1}{\sqrt{1+\abs{2\pi\xi}^{1-s}R^s_\d(\xi)}}-1\Biggr)\abs{2\pi\xi}^{\frac{1-s}{2}},
\] 
which is bounded since $R^s_\d$ (cf.~\eqref{eq:decayR}) is a Schwartz function. A particular consequence is that $\Kcal_\delta^s(-\Delta)^{\frac{1-s}{4}}=(-\Delta)^{\frac{1-s}{4}}\Kcal_\delta^s:H^{\frac{1-s}{2}}(\R^n) \to L^2(\R^n)$ is a bounded linear operator.

 Then,~\eqref{eq:weakpdo} turns into
\[
\int_{\R^n} (-\Delta)^{\frac{1-s}{4}}w\;(-\Delta)^{\frac{1-s}{4}}\phi + \bigl(2\Kcal_\delta^s(-\Delta)^{\frac{1-s}{4}}w +(\Kcal_\delta^s)^2w- \mathbbm{1}_{\Gamma_\d}{g}\bigr)\phi\,dx =0 \quad \text{for all $\phi \in C_c^{\infty}(\Gamma')$},
\]
which implies that $w_\ast$ weakly satisfies 
\[
\begin{cases}
(-\Delta)^{\frac{1-s}{2}}w = 2\Kcal_\delta^s(-\Delta)^{\frac{1-s}{4}}w +(\Kcal_\delta^s)^2w- \mathbbm{1}_{\Gamma_\d}{g}\quad &\text{in $\Gamma'$},\\
w =0 \quad &\text{in $(\Gamma')^c$}.
\end{cases}
\] 
Since the right-hand side of the fractional differential equation lies in $L^2(\Gamma')$ and the boundary $\partial \Gamma'=\partial \Omega' \cup \partial \Omega$ is $C^{1,1}$, we obtain from established regularity results for the fractional Laplacian (see, e.g.,~\cite{Gru15, ViE65, LiP20} for smooth domains and \cite[Theorem~1.1]{AbG23} for $C^{1,1}$-domains) that $w \in H^{1-s}(\R^n)$, as desired. 
\end{proof}

\begin{remark}\label{rem:Grubbreg} 
The range of $p$ for which $w_g \in H^{1-s,p}(\R^n)$ is sharp, since even for smooth $\Omega$, one  can find  $g \in L^p(\Gamma_\d)$ such that $w_g \notin H^{1-s,p}(\R^n)$ when $p \in [\frac{2}{1-s},\infty)$.  To see this, let us take $w$ equal to $\dist(\cdot,\partial \Gamma')^{\frac{1-s}{2}}$ near the boundary of $\Gamma'$ with $\Pcal^s_\d w |_{\Gamma'} \in C^{\infty}(\R^n)|_{\Gamma'}$ as in~Remark~\ref{rem:grubbth} b). Then, we define 
\begin{align*}
\tilde{w}:=\Qcal^s_\d (\mathbbm{1}_{\Omega_\d} \Pcal^s_\d w),
\end{align*} which coincides with $w$ in a neighborhood of $\Omega$, given~\ref{itm:h2}. Therefore, $\tilde{w}$ equals $\dist(\cdot,\partial \Omega)^{\frac{1-s}{2}}$ near the boundary of $\Omega$, which yields $\tilde{w} \notin H^{1-s,p}(\R^n)$. However, since $\tilde{w}=w=0$ in $\Omega$ and $\supp(\tilde{w}) \subset \Omega_{2\d}$, we deduce that
\[
\begin{cases}
\Pcal^s_\d \tilde{w} = \mathbbm{1}_{\Gamma_\d} \Pcal^s_\d w &\text{a.e.~in $\Gamma'$},\\
\tilde{w} = 0 &\text{a.e.~in $(\Gamma')^c$.}
\end{cases}
\]
\color{black} With $g:=\Pcal^s_\d w|_{\Gamma_\d} \in L^p(\Gamma_\d)$, the claim follows.
\end{remark}

\color{black}
Lemma~\ref{le:solpdo2}  in combination with Lemma~\ref{le:equivpdo} 
provide useful information about the solution sets for the 
boundary-value problems (P) and
(C). 
Since the statement of Lemma~\ref{le:solpdo2} is qualitatively different depending on whether $p$ is smaller or larger than the critical value $\frac{2}{1-s}$, we discuss these two cases separately. 

Suppose first that $p\in (1, \frac{2}{1-s})$. Then, for any $c \in \R$ and $g\in L^p(\Gamma_\d)$,
the sets $\Pg(c,g)$ and $\Ccg(c,g)$ are singletons and can be represented as
\begin{align}\label{singletons}
\Pg(c,g) = \{w_{c,g}\} \qquad \text{and } \qquad \Ccg(c,g) = \{h_{c,g}\}, 
\end{align}
where $w_{c,g}\in H^{1-s,p}(\R^n)$ and $h_{c,g} :=(\Pcal_\d^s w_{c,g})|_{\Omega_\d}\in L^p(\Omega_\d)$ are the unique solutions to  (P) and (C), respectively. 

Summarizing these findings, we are now in the position to state the main result of this section concerning the characterization of $\Nzero$.

\begin{theorem} [Characterization of $\Nzero$ for $p\in (1, \frac{2}{1-s})$]\label{th:Ncalchar} 
Let $p\in  (1, \frac{2}{1-s})$ and let $\Omega$ be a bounded $C^{1,1}$-domain.
Then,
\begin{align}\label{reprNth}
\Nzero = \bigcup_{c\in\R,g\in L^p(\Gamma_\delta)} \Ccg(c,g)=\bigcup_{c\in\R,g\in L^p(\Gamma_\delta)} \{h_{c, g}\},
\end{align}
where $h_{c, g}$ for $c \in \R$ and $g\in L^p(\Gamma_\d)$ is the unique solution of \text{\rm (C)}.
\end{theorem}
\begin{proof} 
This follows immediately from \eqref{eq:Nunion}, \eqref{Nzero0} and \eqref{singletons}.
\end{proof}

As a consequence of Theorem~\ref{th:Ncalchar}, 
we obtain that the bounded linear map
\begin{align}\label{homPhi}
\Phi_\d^s: \Nzero\to \R\times  L^p(\Gamma_\delta), \quad h\mapsto \Bigl(\int_{\Omega}  Q_\d^s \ast h \, dx, h|_{\Gamma_\d} \Bigr) 
\end{align}
is bijective. The inverse $(\Phi_\d^s)^{-1}:(c, g)\mapsto h_{c, g}$ for $(c, g)\in \R\times L^p(\Gamma_\d)$ is then bounded as well by the Banach isomorphism theorem, i.e., 
there is a constant $C>0$ such that 
\begin{align*}
 \norm{h_{c,g}}_{L^p(\O_\d)} \leq C\bigl(\norm{g}_{L^p(\Gamma_\delta)} + |c|\bigr) \quad\text{for all $g\in L^p(\Gamma_\delta)$ and $c\in \R$. }
 \end{align*}

The discussion above implies that the functions with zero nonlocal gradient are uniquely determined by their values in the single collar $\Gamma_\d$ and an averaging condition involving the kernel function $Q_\d^s$. Besides, one can also observe a one-to-one correspondence between functions in $\Nzero$ and these two basic characteristics: the boundary values in $\Gamma_\d$ and the mean value in $\Omega$. 
Indeed, another isomorphism between $\Nzero$ and $\R\times L^p(\Gamma_\d)$ is given by
\begin{align}\label{homPsi}
\Psi_\d^s:\Nzero\to \R\times  L^p(\Gamma_\delta), \quad h\mapsto \Bigl(\int_{\Omega} h \, dx, h|_{\Gamma_\d} \Bigr),
\end{align}
which follows essentially from the next proposition.
\begin{proposition}[Uniqueness in $\Nzero$ with vanishing mean value]\label{prop:mean0}
 Let $p\in (1, \infty)$ and $\Omega$ be a bounded $C^{1,1}$-domain. If $h\in \Nzero$ satisfies
 \begin{align*}
  h=0  \text{ a.e.~in $\Gamma_\d$ \quad  and \quad $\int_{\Omega} h\, dx=0$, }
  \end{align*}
 then $h=0$. 
\end{proposition}

\begin{proof}
If $h\in \Nzero$ with $h=0$ a.e.~in $\Gamma_\d$, then~\eqref{reprNth} implies the existence of a $c\in \R$ such that $h=h_{c, 0}\in L^p(\Omega_\d)$. If $p \in [\frac{2}{1-s},\infty)>2$, it is clear that $h\in L^2(\Omega_\d)$. Otherwise, this property follows from 
\[
h_{c,0} \in \Ccg(c,0) = \Cfrak^{s,2, \d}(c,0)\subset L^2(\Omega_\d),
\]
see~\eqref{Cindependencep} below. We exploit $\int_\Omega h\,dx=0$ and $\supp(Q_\d^s)\subset B_\d(0)$ to find
\begin{align*}
\norm{\widehat{\mathbbm{1}_{\Omega_\d}h}}_{L^2(\R^n)} 
&= \norm{h}_{L^2(\Omega)}= \normlr{h-Q^s_\d \ast h - \abs{\O}^{-1}\int_\Omega h-Q^s_\d \ast h\,dx}_{L^2(\Omega)} \\ 
&\leq \norm{h-Q^s_\d*h}_{L^2(\Omega)} \leq \norm{\mathbbm{1}_{\Omega_\d}h-Q^s_\d*(\mathbbm{1}_{\Omega_\d}h)}_{L^2(\R^n)} \leq \norm{(1-\widehat{Q}^s_\d)\widehat{\mathbbm{1}_{\Omega_\d}h}}_{L^2(\R^n)}.
\end{align*}
Since $0 \leq 1-\widehat{Q}^s_\d <1$ (see~\eqref{int=1} and Remark~\ref{rem:scaling}\,b)), we deduce that $\widehat{\mathbbm{1}_{\Omega_\d}h}=0$, and hence, $h=0$, as stated.
\end{proof}
\begin{remark}[Uniqueness in $\Nzero$ with enlarged single layer]\label{rem:enlargedsinglelayer} 
The mean value condition in the previous proposition can be removed in exchange for replacing the Dirichlet condition in the single layer $\Gamma_\d$ by a Dirichlet condition in $\Omega_\d \setminus  \overline{O}$ for any $O \Subset \Omega$.
This means, if $h \in \Nzero$ satisfies $h=0$ a.e.~in $\Omega_\d \setminus \overline{O}$, then $h=0$. 

 Indeed, let $O'\subset \Omega$ be smooth with $O\Subset O' \Subset \Omega$ and take $\eta \in L^1(\R^n)$ with $\norm{\eta}_{L^{1}(\R^n)}=1$ and $\supp(\eta) \subset B_\eps(0)$ for $\eps >0$ sufficiently small. Since there is a $c \in \R$ such that $Q^s_\d *h=c$ a.e.~in $\O$ with $h=0$ a.e.~in $\O_\d \setminus \overline{O}$, we find that the convolution $h_{\eta}:=\eta*h$ satisfies  \begin{center} $Q^s_\d *h_{\eta}=c$ a.e.~in $O'$ with $h_{\eta}=0$ a.e.~in $O'_\d \setminus \overline{O'}$. \end{center} By the uniqueness statements in Theorem~\ref{th:Ncalchar} and Proposition~\ref{th:Ncalchar_plarge} (applied to the set $O'$), it follows that $h=h_\eta$ for all such $\eta$, which is only possible for $h=0$. 
\end{remark}

\begin{remark}[Equivalent norms on $\Nzero$]\label{rem:normonN} 
Based on the isomorphisms $\Psi_\d^s$ and $\Phi_\d^s$ we conclude that defining
\begin{align*}
\vertiii{h}_{\Nzero}:=\norm{h}_{L^p(\Gamma_\d)} + \Bigl|\int_\Omega Q^s_\d * h\, dx\Bigr| 
\end{align*}
and 
\begin{align*}
\vertiii{h}_{\Nzero}:=\norm{h}_{L^p(\Gamma_\d)} + \Bigl|\int_\Omega h\, dx\Bigr| 
\end{align*}
for $h\in \Nzero$ yields two norms on $\Nzero$ that are equivalent to $\norm{\cdot}_{L^p(\O_\d)}$ when $p\in (1,\frac{2}{1-s})$. 
\end{remark}

Finally, we study the case $p \in [\frac{2}{1-s},\infty)$. Before stating the corresponding representation result for $\Nzero$, we collect a few further observations about the solutions to (C) and (P). Suppose in the following that $c \in \R$ and $g\in L^p(\Gamma_\d)$. By Lemma~\ref{le:solpdo2}, the solution set $\Pg(c,g)$ has at most cardinality $1$; specifically, 
 \begin{align*}
\Pg(c,g) = \{w_{c,g}\}\cap H^{1-s, p}(\R^n). 
\end{align*}
Since $L^p(\Omega_\d)\subset L^q(\Omega_\d)$ for all $1<q\leq p$,  Lemma~\ref{le:solpdo2} implies also that $w_{c,g}\in H^{1-s, q}(\R^n)$ for all $q\in (1, \frac{2}{1-s})$. In particular, this shows that the solution sets of (P) are independent of the integrability parameters, in the sense that
\begin{align*}
\Pfrak^{s,q, \d}(c,g) = \Pfrak^{s,2, \d}(c,g) \quad \text{for all $q\in \textstyle (1, \frac{2}{1-s})$,}
\end{align*} 
so that we can conclude
\begin{align*}
\Pg(c,g) =  \Pfrak^{s,2, \d}(c,g) \cap H^{1-s, p}(\R^n).
\end{align*}

Considering the one-to-one relation between the solutions of (C) and (P) (see Lemma~\ref{le:equivpdo}), the above properties of $\Pg(c,g)$ carry over to $\Ccg(c,g)$. Hence,
\begin{align}\label{Cindependencep}
 \Cfrak^{s,q, \d}(c,g)=\Cfrak^{s,2,\d}(c,g)\quad \text{for all $q\in \textstyle (1, \frac{2}{1-s})$,}
\end{align}
and along with~\eqref{eq:psdhom}, 
 \begin{align}\label{Ccap}
\Ccg(c,g) =  \Cfrak^{s,2, \d}(c,g) \cap L^p(\Omega_\d). 
\end{align}
The latter gives rise to the next result. 

\begin{proposition} [Characterization of $\Nzero$ for $p \in [\frac{2}{1-s},\infty)$]\label{th:Ncalchar_plarge}
Let $p \in [\frac{2}{1-s},\infty)$ and let $\Omega$ be a bounded $C^{1,1}$-domain. 
Then,
\begin{align*}
\Nzero =N^{s, 2,\d}(\Omega)\cap L^p(\Omega_\d),
\end{align*}
where $N^{s,2,\d}(\Omega)$ can be characterized as in Theorem~\ref{th:Ncalchar}. 
\end{proposition}

\begin{proof}
The combination of \eqref{eq:Nunion}, \eqref{Nzero0}, and \eqref{Ccap} proves the claim.
\end{proof}

 \begin{remark}\label{rem:plargebcs}
The maps  $\Phi_\d^s$ and $\Psi_\d^s$ from \eqref{homPhi} and~\eqref{homPsi} can be defined analogously when $p \in [\frac{2}{1-s},\infty)$. While Proposition~\ref{th:Ncalchar_plarge} shows that they are still injective, surjectivity generally fails in view of Remark~\ref{rem:Grubbreg}. This shows that not all boundary values in $L^p(\Gamma_\d)$ can be attained by functions in $\Nzero$, in contrast to the case $p\in (1, \frac{2}{1-s})$.
\end{remark}

\color{black}
\subsection{Regularity properties of functions with zero nonlocal gradient and examples}\label{subsec:regulex}
In this section, we dive deeper into some of the properties of functions with zero nonlocal gradient, such as their regularity, and we will show some numerical examples to illustrate how they generally behave.

We start off by showing that all functions in $\Nzero$ are smooth inside $\Omega$.

\begin{corollary}[Functions with vanishing nonlocal gradient are smooth in $\Omega$]
 Let $p\in (1, \infty)$, then every $h\in \Nzero$ satisfies $$h|_\Omega\in C^\infty(\Omega).$$
 \end{corollary}

\begin{proof}
It suffices to prove the statement for $h \in \Ccg(0,g)$ with $g \in L^p(\Gamma_\d)$, given~\eqref{eq:Nunion} and the fact that $\Ccg(c,g)=\Ccg(0,g-c)+c$ for all $c \in \R$. If $h\in \Ccg(0,g)$, we deduce from~Lemma~\ref{le:equivpdo} that 
\[
h=(\Pcal^s_\d w)|_{\Omega_{\delta}} \quad \text{for a $w \in \Pg(0,g)\subset H^{1-s, p}_0(\Gamma')$.} 
\] 
To see that the restriction $(\Pcal_\delta^s w)|_\Omega$ is smooth, we argue as follows. For any $\eps>0$ sufficiently small and $\psi\in C^\infty(\R^n)$ with $\psi =\Div V_\d^s$ on $B_\eps(0)^{c}$, it holds that
\begin{align*}
\Pcal^s_\d w = \Div V_\delta^s \ast w= \psi*w \quad \text{on $\R^n\setminus (\supp (w))_\eps$;}
\end{align*} 
this follows from~\eqref{eq:vsdrepr} via integration by parts, along with an approximation argument. Consequently, $\Pcal_\delta^s w$ can be expressed as the convolution  of a compactly supported $L^p$-function with a smooth function on $\R^n\setminus (\Omega^c)_\eps$ for any $\eps$, and is thus smooth on the union of all these sets, which is $\Omega$.
\end{proof}

In general, we do not expect that functions in $\Nzero$ will be regular on the larger domain $\Omega_\d$ given Remark~\ref{rem:Grubbreg}, see also~Figure~\ref{fig:zoom} and \ref{fig:Nzero} below. However, there do exist smooth non-constant functions in $\Nzero$ (cf.~Proposition~\ref{prop:nonconstant}) and they are exactly those that can be obtained from the translation mechanism.

\begin{proposition}[Functions in $\Nzero$ with extra regularity]\label{prop:extraregularity}
Let $p \in [1,\infty]$, then it holds that
\[
N^{s,p,\d}(\Omega) \cap H^{s,p,\d}(\R^n)|_{\Omega_\d} = \{\Pcal^s_\d v|_{\Omega_\d} \,:\, v \in W^{1,p}(\R^n) \ \text{with} \ \nabla v = 0 \ \text{a.e.~on $\Omega$}\},
\]
and
\[
N^{s,p,\d}(\Omega) \cap C_c^{\infty}(\R^n)|_{\Omega_\d}= \{\Pcal^s_\d v|_{\Omega_\d} \,:\, v \in C_c^{\infty}(\R^n) \ \text{with} \ \nabla v = 0 \ \text{on $\Omega$}\}.
\]
\end{proposition}
\begin{proof}
Let $u \in N^{s,p,\d}(\Omega) \cap H^{s,p,\d}(\R^n)|_{\Omega_\d}$ and consider an extension $\tilde{u} \in H^{s,p,\d}(\R^n)$ of $u$. Then, we find by Lemma~\ref{le:translation1} that $v:=\Qcal^s_\d \tilde{u} \in W^{1,p}(\R^n)$ with $\nabla v = D^s_\d u = 0$ a.e.~on $\Omega$. Moreover, by Lemma~\ref{le:translation2} we find that
\[
\Pcal^s_\d v|_{\Omega_\d} = \tilde{u}|_{\Omega_\d} = u,
\]
as desired. On the other hand, if $u = \Pcal^s_\d v|_{\Omega_\d}$ for $v \in W^{1,p}(\R^n)$ with $\nabla v = 0$ a.e.~on $\Omega$, then $u \in H^{s,p,\d}(\R^n)|_{\Omega_\d}$ and $D^s_\d u = \nabla \Qcal^s_\d u = \nabla v =0$ a.e.~on $\Omega$. This proves the first identification.

For the smooth case we can argue in the same way, by also using that $\Qcal^s_\d$ and $\Pcal^s_\d$ map smooth functions to smooth functions (cf.~\eqref{eq:vsdrepr}). Note that $\Pcal^s_\d$ might not preserve the compact support, but this is not an issue since $C_c^{\infty}(\R^n)|_{\Omega_\d}=C^{\infty}(\R^n)|_{\Omega_\d}$.
\end{proof}

We close this section with an illustration of selected one-dimensional examples of functions with zero nonlocal gradient. Figure~\ref{fig:zoom} depicts a numerical approximation of the unique function $h_{c,g}\in \Cfrak^{s,2,\d}(c,g)$ with $c=0$ and $g \equiv -1$ on $\Gamma_\d$. 
While  $h_{c,g} \in L^q(\Omega_\d)$ for all $q \in (1,\frac{2}{1-s})$ according to \eqref{Cindependencep}, we see in the first plot that this function has a jump singularity at the boundary of the domain $\Omega=(-3,3)$. This indicates that $h_{c, g}$ might not lie in $L^p(\Omega_\d)$ for all $p \in (1,\infty)$, reflecting the observations from Remark~\ref{rem:Grubbreg} and \ref{rem:plargebcs}. Moreover, while one may expect from the first illustration that the function is constant on a sub-interval of $(-3,3)$, the enlarged plots show that this is not the case. Indeed, $h_{c,g}$ seems to be displaying oscillations with decreasing amplitude, which is in line with the fact that functions in $\Nzero$ need not be constant on any subset of $\Omega$ (cf.~Proposition~\ref{prop:nonconstant}). It is an interesting topic for further study to understand these oscillatory patterns better, and to see if all non-constant functions in $\Nzero$ have a similar behavior.
\begin{figure}[h]
\noindent\makebox[\textwidth]{
\begin{tabular}{c c}
\includegraphics[scale=0.47]{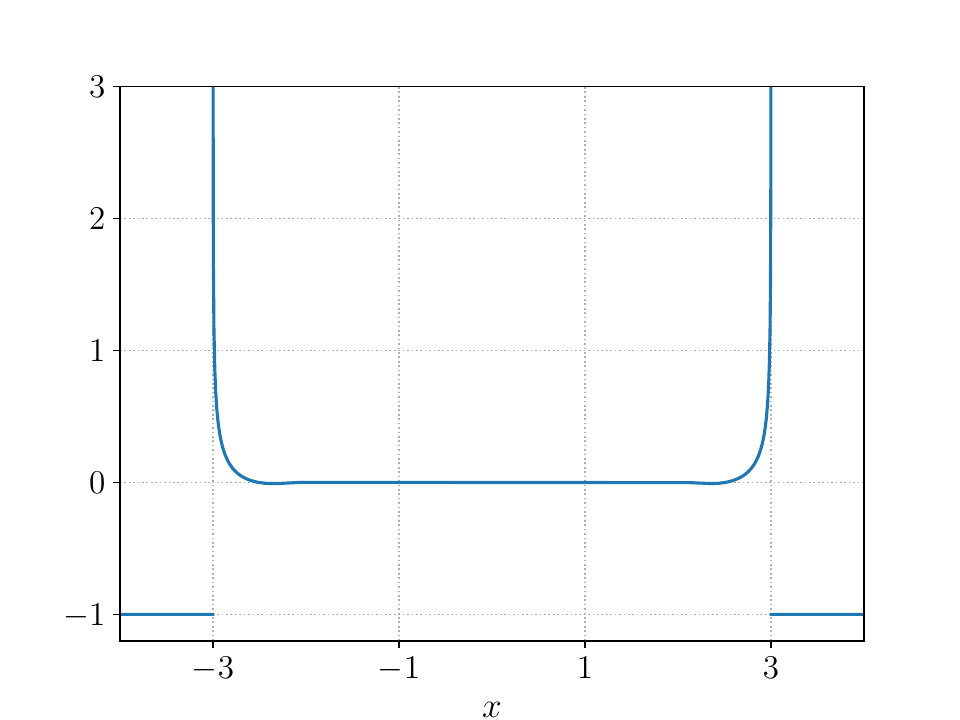} & \includegraphics[scale=0.47]{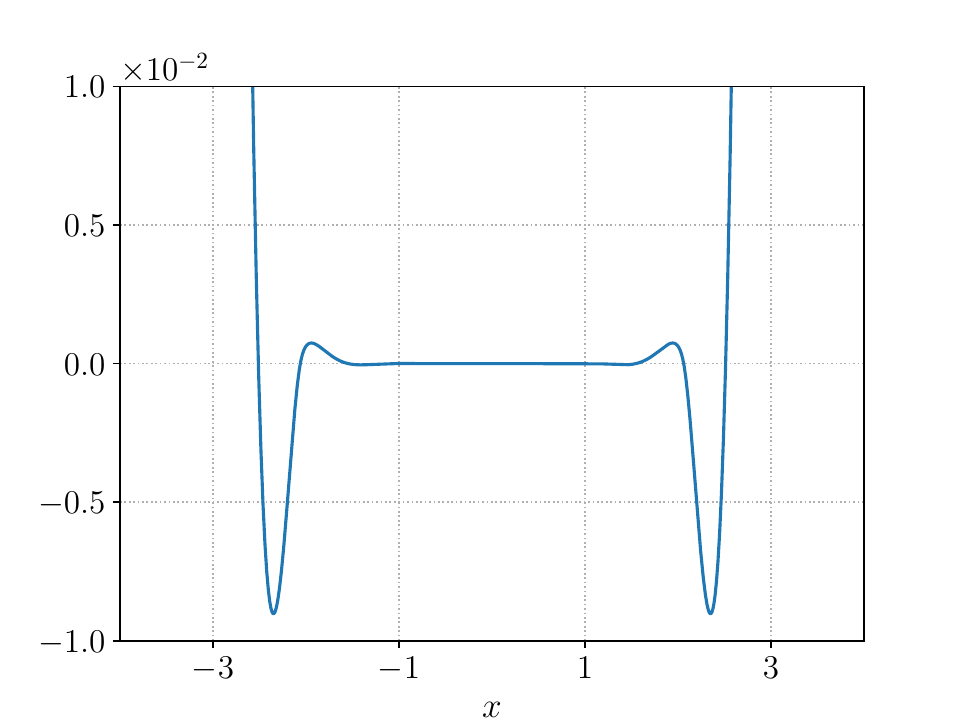}  \\
\includegraphics[scale=0.47]{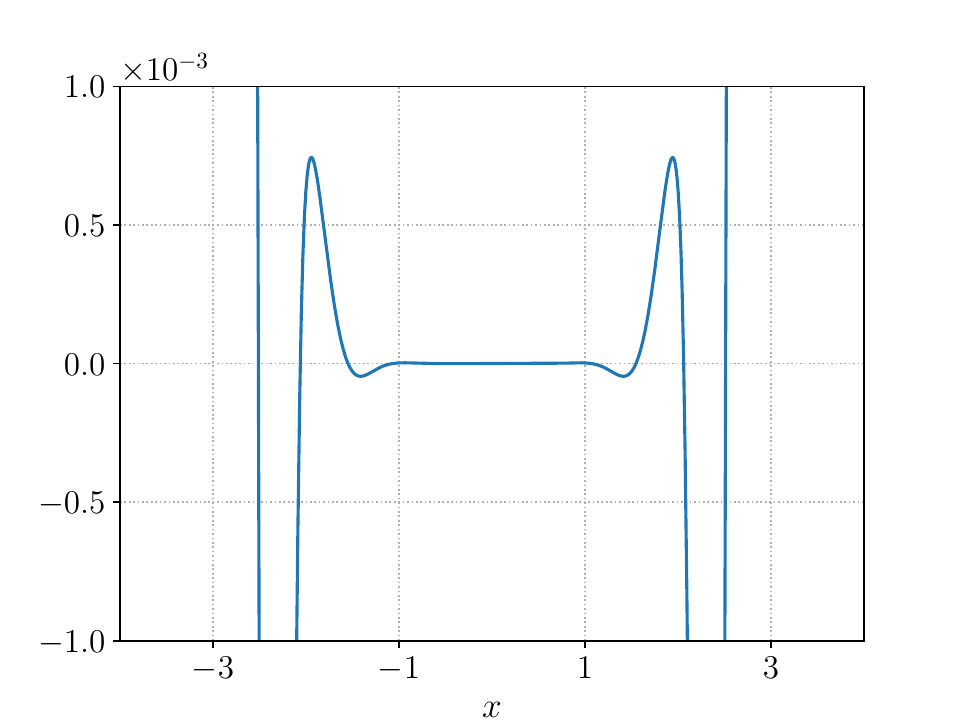} & \includegraphics[scale=0.47]{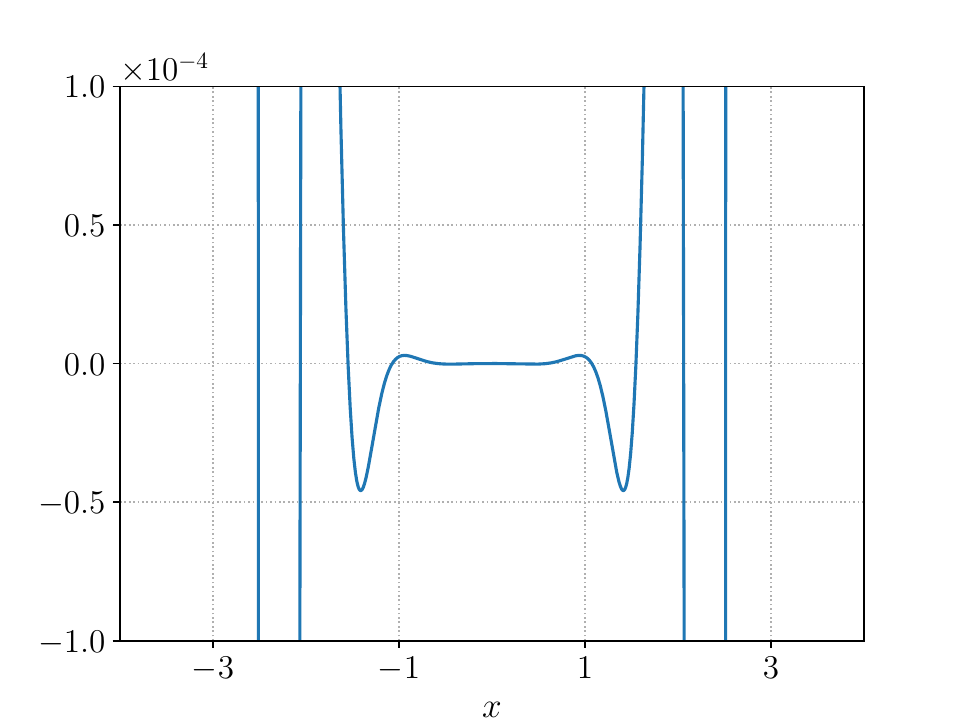}
\end{tabular}}
\caption{Numerical approximation of the function $h_{c,g} \in N^{s,2,\d}(\Omega)$ for $c=0$ and $g\equiv -1$ on $\Gamma_\d$ with increasing degrees of zoom. The parameters for the computation are $n=1$, $\Omega=(-3,3)$, $s=\frac{1}{2}$, $\d=1$ and $w_\d \in C_c^{\infty}(-1,1)$ is a bump function equal to $1$ on $(-\frac{1}{2},\frac{1}{2})$.}\label{fig:zoom}
\end{figure}

In Figure~\ref{fig:Nzero}, there are two further examples of functions with zero nonlocal gradient. The left-hand example is similar to the one from Figure~\ref{fig:zoom}, but with different boundary values. It still features jump singularities at the boundary, and is nearly constant away from the boundary. The right-hand example in Figure~\ref{fig:Nzero} shows a function with zero nonlocal gradient  constructed via the characterization in Proposition~\ref{prop:extraregularity}. In contrast to the other examples, this one does not have a jump singularity at the boundary. By construction, it is smooth and an element of $\Nzero$ for all $p \in [1,\infty]$.

\begin{figure}[h!]
\noindent\makebox[\textwidth]{
\begin{tabular}{c c}
\includegraphics[scale=0.47]{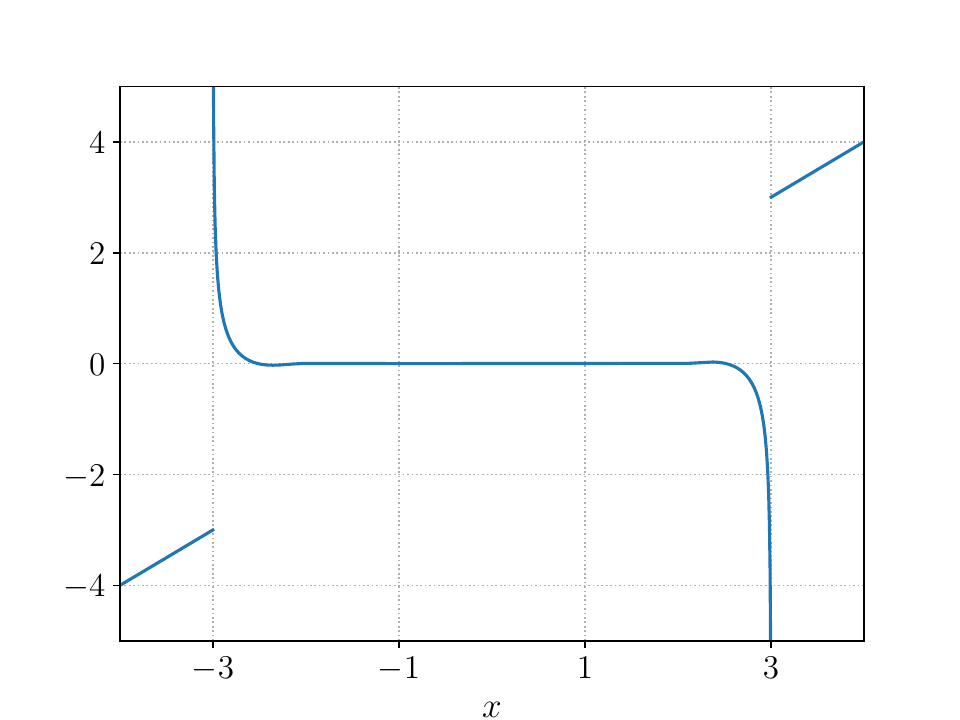} & \includegraphics[scale=0.47]{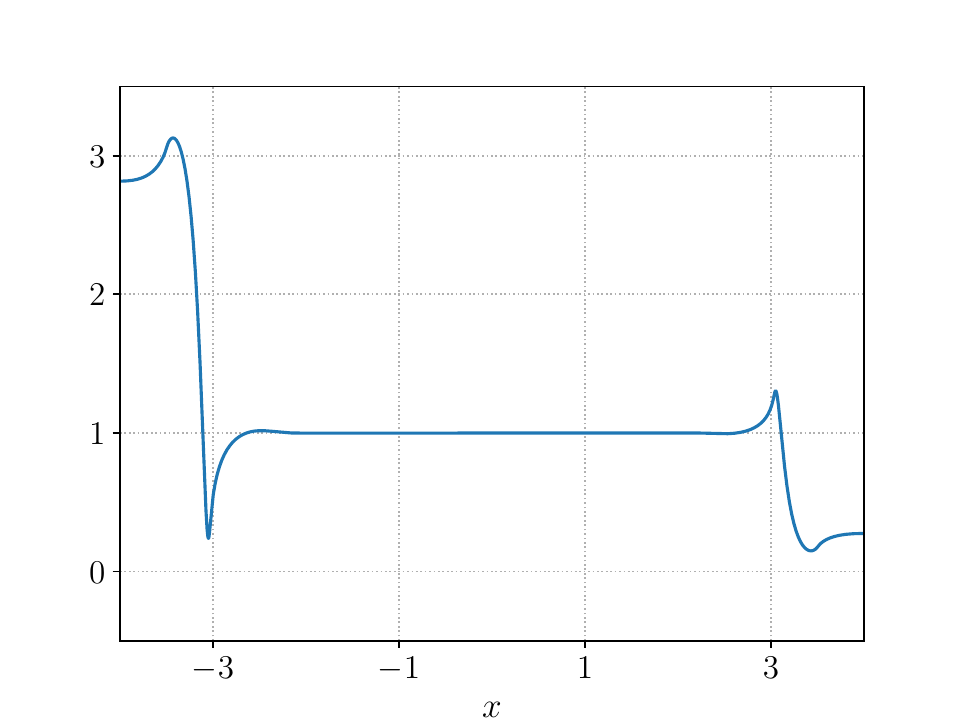}
\end{tabular}}
\caption{Left: A numerical approximation of the function $h_{c,g} \in N^{s,2,\d}(\Omega)$ with $c=0$ and $g(x)=x$ for $x \in \Gamma_\d$. Right: A plot of $\Pcal^s_\d v|_{\Omega_\d} \in \bigcap_{p \in [1,\infty]}\Nzero$ with $v(x)=1+5\phi(x+4)-2\phi(x-4)$ for a non-negative bump function $\phi \in C_c^{\infty}(-1,1)$. The parameters are the same as in Figure~\ref{fig:zoom}.}\label{fig:Nzero}

\end{figure}

\section{Technical tools involving functions with zero nonlocal gradient}\label{sec:tools}

In this section, we present several results regarding the function spaces $H^{s,p,\d}(\O)$ in which the set $\Nzero$ plays an important role. We start off with a bounded-domain analogue of the isomorphism between $H^{s,p,\d}(\R^n)$ and $W^{1,p}(\R^n)$ from~\cite[Section~2.4]{CKS23} that turns nonlocal gradients into gradients. Subsequently, we study extensions of functions in $H^{s, p, \delta}(\Omega)$ to the whole space $\R^n$ and prove new nonlocal Poincar\'{e}  and Poincar\'{e}-Wirtinger inequalities and compactness results.

\subsection{Connection between classical and nonlocal Sobolev spaces}\label{sec:app1}
As we know from Section~\ref{subsec:translation}, the translation operators $\Qcal_\delta^s$ and its inverse $\Pcal^s_\d$ provide a isomorphism between the spaces $H^{s,p,\d}(\R^n)$ and $W^{1,p}(\R^n)$  with the properties $\nabla \circ \Qcal^s_\d=D_\delta^s$ and $D_\delta^s\circ \Pcal_\delta^s =\nabla$. On a bounded open set $\Omega \subset \R^n$, it still holds for $u \in H^{s,p,\d}(\Omega)$ that $\Qcal^s_\d u$ lies in $W^{1,p}(\Omega)$ with $\nabla (\Qcal^s_\d u) = D^s_\d u$ (cf.~Lemma~\ref{le:translation1}), but $\Pcal^s_\d$ is not defined on $W^{1,p}(\Omega)$, which prevents an identification with $H^{s, p,\d}(\Omega)$ in analogy to the setting on the whole space $\R^n$. 

It turns out that one can resolve this issue and find  a perfect translation mechanism also between the classical and nonlocal Sobolev spaces on bounded sets, by considering the spaces modulo the functions with zero (nonlocal) gradient. In this spirit, our next theorem gives a natural generalization of Lemma~\ref{le:translation1} and~\ref{le:translation2}, cf.~also~\cite[Section~2.4]{CKS23}.

To state the result precisely, let us introduce the quotient spaces
\[
\widetilde{H}^{s,p,\d}(\Omega):=H^{s,p,\d}(\Omega)/\Nzero \quad \text{and}\quad \widetilde{W}^{1,p}(\Omega):=W^{1,p}(\Omega)/\Ccal(\Omega),
\]
where $\Ccal(\Omega):=\{v \in L^p(\Omega)\,:\, v \ \text{is constant}\}$; 
for the equivalence classes in $\widetilde{H}^{s,p,\d}(\Omega)$, we write $[u]_\delta^s=u+\Nzero$ with a representative in $u\in H^{s,p, \d}(\Omega)$, and analogously, $[v]=v+\Ccal(\Omega)$ with $v\in W^{1,p}(\Omega)$ for elements in $\widetilde{W}^{1,p}(\Omega)$.  
We endow these spaces with the norms given by
\begin{align}\label{norms}
\normb{[u]_\delta^s}_{\widetilde{H}^{s,p, \d}(\Omega)} :=  \norm{D_\delta^su}_{L^p(\Omega;\R^n)} \quad \text{and}\quad
\normb{[v]}_{\widetilde{W}^{1,p}(\Omega)} :=  \norm{\nabla v}_{L^p(\Omega;\R^n)}, 
\end{align}
noting that $D^s_\d u$ and $\nabla v$ are both independent of the chosen representative of $[u]_\delta^s$ and $[v]$, respectively. Moreover, let $\widetilde{D}^s_\d [u]_\delta^s := D^s_\d u$ and $\widetilde{\nabla} [v]:= \nabla v$ for $u\in H^{s,p,\d}(\Omega)$ and $v \in W^{1,p}(\Omega)$, respectively, where the choice of representative is irrelevant. 

\begin{theorem}[Isomorphism between $\Hquo$ and $\Wquo$] \label{th:connecquo}
Let $p \in [1,\infty]$. Then, the linear map
\begin{align*}
&\widetilde{\Qcal}_\delta^s:\Hquo \to \Wquo, \quad  [u]_\delta^s\mapsto [\Qcal_\delta^s u]
\end{align*}
defines a isometric isomorphism, and it holds with $\widetilde\Pcal_\delta^s:=(\widetilde\Qcal_\delta^s)^{-1}$ that
 \begin{align}\label{identities_quotientspace}
 \widetilde{\nabla}  \circ \widetilde \Qcal_\delta^s = \widetilde{D}_\delta^s \quad \text{and}\quad  \widetilde{D}_\delta^s  \circ \widetilde \Pcal_\delta^s = \widetilde{\nabla}.
 \end{align}
\end{theorem}

\begin{proof}
Note first that $\widetilde{\Qcal}_\d^s$ is well-defined since $\Qcal^s_\d h$ is constant for any $h \in \Nzero$. The first identity in~\eqref{identities_quotientspace} follows immediately from $\nabla \circ \Qcal^s_\d=D_\delta^s$, and we can compute that 
\[
\norm{\widetilde{\Qcal}_\delta^s [u]_\delta^s}_{\Wquo} = \norm{\nabla (\Qcal^s_\d u)}_{L^p(\Omega;\R^n)} = \norm{D^s_\d u}_{L^p(\Omega;\R^n)}  =  \norm{[u]_\delta^s}_{\widetilde{H}^{s,p,\d}(\Omega)}
\]
for all $u\in H^{s, p, \delta}(\Omega)$, which shows that $\widetilde{\Qcal}_\delta^s$ is an isometry. 
To prove the bijectivity, we claim that the inverse of $\widetilde{\Qcal}_\d^s$ is given by 
\[
\widetilde{\Pcal}_\d^s[v] =\big[ \Pcal^s_\d(\Ecal v)|_{\Omega_\d} \big]_\delta^s \quad \text{for $v\in W^{1,p}(\Omega)$,}
\]
where $\Ecal:W^{1,p}(\Omega) \to W^{1,p}(\R^n)$ is any bounded linear extension operator. Indeed, it holds that $$D^s_\d(\widetilde{\Pcal}_\d^s[v]) =D_\delta^s (\Pcal_\delta^s(\Ecal v)|_{\Omega_\d}) = \nabla (\Ecal v) |_{\Omega}=\nabla v \quad\text{ for $v\in W^{1,p}(\Omega)$},$$ 
from which we infer the second part of~\eqref{identities_quotientspace}, as well as $\widetilde{\Pcal}_\delta^s \circ \widetilde{\Qcal}_\delta^s = \Id$ and $ \widetilde{\Qcal}_\delta^s\circ \widetilde{\Pcal}_\delta^s=\Id$.

\end{proof}

\begin{remark}
\label{rem:quotientnorm}
The boundedness of $\widetilde{\Qcal}^s_\d$ and $\widetilde{\Pcal}^s_\d$ holds as well, if $\widetilde{H}^{s, p, \delta}(\Omega)$ is equipped with the associated quotient norm, i.e.,
\[
\vertiii{[u]_\delta^s}_{\widetilde{H}^{s,p,\d}(\Omega)} :=
\inf_{h\in \Nzero} \norm{u-h}_{L^p(\Omega_\delta)}+  \norm{D^s_\d u}_{L^p(\Omega;\R^n)} \]
for $u\in H^{s, p, \d}(\Omega)$. Indeed, for $\widetilde{\Qcal}^s_\d$ this is clear, whereas for $\widetilde{\Pcal}^s_\d$ we can compute for $v \in W^{1,p}(\Omega)$ with $\int_\Omega v\,dx=0$ that
\begin{align*}
\vertiii{\widetilde{\Pcal}^s_\d[v]}_{\widetilde{H}^{s,p,\d}(\Omega)} &\leq \norm{\Pcal^s_\d(\Ecal v)}_{L^p(\Omega_\d)} + \norm{\nabla v}_{L^p(\Omega;\R^n)} \\
&\leq C\norm{\Ecal v}_{W^{1,p}(\R^n)} + \norm{\nabla v}_{L^p(\Omega;\R^n)} \\
&\leq C\norm{v}_{W^{1,p}(\Omega)} + \norm{\nabla v}_{L^p(\Omega;\R^n)} \leq C\norm{[v]}_{\widetilde{W}^{1,p}(\Omega)},
\end{align*}
where the second inequality uses Lemma~\ref{le:translation2}, and the last the classical Poincar\'{e}-Wirtinger inequality. Moreover, for $p \in (1,\infty)$, the operator norm of $\widetilde{\Pcal}^s_\d$ is independent of $s$ by \eqref{eq:uniformpbound}. We use this observation later in Corollary~\ref{cor:poincareN} to deduce a new nonlocal Poincar\'{e}-Wirtinger equation. 

In the classical Sobolev setting, 
the norm $\norm{\cdot}_{\widetilde{W}^{1,p}(\Omega)}$ in~\eqref{norms} is equivalent to the quotient norm on $\widetilde{W}^{1,p}(\Omega)$ by the standard Poincar\'e-Wirtinger inequality.
\end{remark}

If the characterization of $\Nzero$ in Theorem~\ref{th:Ncalchar} holds, then any boundary values can be attained in the layer $\Gamma_\d$ by elements in an equivalence class of $\widetilde{H}^{s, p, \delta}(\Omega)$. In other words,  for each $u\in H^{s,p, \d}(\Omega)$ and each $g\in L^p(\Gamma_\d)$, there exists a representative of $[u]_\delta^s = u+\Nzero$ that coincides with $g$ in $\Gamma_\d$.
Based on this observation, we can state the following consequence of Theorem~\ref{th:connecquo}. 
\begin{corollary}
Let $p\in (1, \frac{2}{1-s})$ and $\Omega$ be a bounded $C^{1,1}$-domain. Then, for every $v \in W^{1,p}(\Omega)$ and $g \in L^p(\Gamma_\d)$, there is a $u \in H^{s,p,\d}(\Omega)$ such that
\[
\begin{cases}
D^s_\d u = \nabla v \quad &\text{a.e.~in $\Omega$},\\
u = g \quad &\text{a.e.~in $\Gamma_\d$}.
\end{cases}
\]
\end{corollary}
\begin{proof}
Let $v\in W^{1,p}(\Omega)$. Then, Theorem~\ref{th:connecquo} implies that $\nabla v= \widetilde{\nabla}[v] = \widetilde{D}_\delta^s[u]_\delta^s = D_\delta^s u$ for some $u\in H^{s,p,\d}(\Omega)$. By Theorem~\ref{th:Ncalchar}, we may assume that $u$ coincides with $g$ in the boundary layer $\Gamma_\d$,  which yields the desired function.
\end{proof}

\subsection{Extension modulo functions with zero nonlocal gradient.}\label{subsec:extension}
While not every function in $H^{s, p,\delta}(\Omega)$ is the restriction of a function in $H^{s, p, \delta}(\R^n)$ (cf.~Proposition~\ref{prop:Nzeroinfinitedim}), 
we can show nevertheless that extensions to the whole space $\R^n$ are possible up to function with  zero nonlocal gradient.
 This technical tool has several applications within this paper. It has appeared already in the proof of Proposition~\ref{prop:Nzeroinfinitedim}, where it provided an efficient way for generating functions in $\Nzero$. 

With $p\in [1, \infty]$, we define for a given bounded linear extension operator $\Ecal:W^{1,p}(\Omega) \to W^{1,p}(\R^n)$, 
\begin{align}\label{extensionmodulo}
\Ecal^s_\d:H^{s,p,\d}(\Omega) \to H^{s,p,\d}(\R^n), \quad \Ecal^s_\d := \Pcal^s_\d \circ \Ecal\circ \Qcal^s_\d,
\end{align} 
with the translation operators $\Qcal_\d^s$ and $\Pcal_\d^s$ from~Section~\ref{subsec:translation}. 
As the composition of bounded linear operators, $\Ecal_\delta^s$ is bounded, even uniformly in $s$ when $p\in (1, \infty)$,  see~\eqref{eq:uniformpbound}. 
In view of \eqref{translation_formula}, we infer for every $u\in H^{s, p,\delta}(\Omega)$ that
 $D^s_\d \Ecal^s_\d u = D^s_\d u$ on $\Omega$, and thus,
\begin{align*}
u-\Ecal^s_\d u|_{\Omega_\d} \in \Nzero \quad \text{for $u\in H^{s, p, \delta}(\Omega)$.}
\end{align*} 
\color{black}In this sense, $\Ecal^s_\d$ can be viewed as an extension operator on $H^{s,p,\d}(\Omega)$ modulo functions in $\Nzero$. 

Note further that $\Ecal_\d^s$, as a map from $H^{s,p,\d}(\Omega)$ to $L^p(\Omega_\d)$, is compact for $p \in (1,\infty)$ due to the compact embedding of $H^{s,p,\d}(\R^n)=H^{s,p}(\R^n)$ into $L^p(\Omega_\d)$, see~Section~\ref{sec:sobolevspaces}. Thus, if $(u_j)_j$ is a weakly convergent sequence in $H^{s,p, \delta}(\Omega)$ with limit $u\in H^{s, p, \d}(\Omega)$, then 
\begin{align}\label{conv:extensionmodulo}
\Ecal^s_\d u_j \to \Ecal^s_\d u \quad \text{ in $L^p(\Omega_\d)$}.
\end{align}

\subsection{A new nonlocal Poincar\'{e} inequality } 
Another application of Theorem~\ref{th:Ncalchar} is that we can derive a new Poincar\'{e} inequality for the nonlocal gradient. As opposed to the Poincar\'{e} inequality in \cite[Theorem~6.1]{BCM23}, which requires functions to be zero in the double collar $\Gamma_{\pm\d}$, the new one only imposes a condition in $\Gamma_\d$ together with an average-value condition. Precisely, will work with functions in the linear subspace
\begin{align*}
 \mathring{H}^{s,p,\d}(\Omega):=\bigl\{u\in H^{s, p, \d}(\Omega): u=0 \text{ a.e.~in $\Gamma_\d$}, \ \textstyle
  \int_\Omega u\, dx=0\bigr\}.
 \end{align*}

\color{black}
\begin{theorem}[Nonlocal Poincar\'{e} inequality]\label{th:poincspecific}
Let $p \in (1,\frac{2}{1-s})$ and $\Omega \subset \R^n$ be a bounded $C^{1,1}$-domain.
Then, there exists a constant $C>0$ such that
\[
\norm{u}_{L^p(\Omega_\d)} \leq C\norm{D^s_\d u}_{L^p(\Omega;\R^n)} 
\]
for all $u \in \mathring{H}^{s,p,\d}(\Omega)$. 
\end{theorem}

\begin{proof} 
 The proof stragegy follows a well-known contradiction argument, with Lemma~\ref{le:comp} below as  main technical ingredient. Suppose  the statement is false, then there is a sequence $(u_j)_j \subset \mathring{H}^{s,p,\d}(\Omega)$ 
with $\norm{u_j}_{L^p(\Omega_\d)} > j \norm{D^s_\d u_j}_{L^p(\Omega;\R^n)}$ for all $j$. By defining the sequence $(\tilde{u}_j)_j\subset \mathring{H}^{s,p, \d}(\Omega)$ via
\[
\tilde{u}_j:=\frac{u_j}{\norm{u_j}_{L^p(\Omega_\d)}},
\]
we obtain $\norm{\tilde{u}_j}_{L^p(\Omega_\d)}=1$ and $\norm{D^s_\d \tilde{u}_j}_{L^p(\Omega;\R^n)} \leq 1/j$ for each $j$. 
This allows us to conclude for a non-relabeled subsequence that
\begin{align*}
\tilde{u}_j \weakto h \quad \text{in $H^{s,p,\d}(\Omega)$ as $j\to \infty$,}
\end{align*} 
with a limit function $h\in H^{s,p,\d}(\Omega)$ that satisfies $D^s_\d h=0$, or in other words, $h\in \Nzero$.  Due to the weak closedness of $\mathring{H}^{s,p,\d}(\Omega)$, we also find that $u \in \mathring{H}^{s,p,\d}(\Omega)$, which yields $h=0$ by Proposition~\ref{prop:mean0}.

Finally, we infer from Lemma~\ref{le:comp} that $\tilde{u}_j \to 0$ in $L^p(\Omega_\d)$ as $j\to \infty$, which contradicts $\norm{\tilde{u}_j}_{L^p(\Omega_\d)}=1$ for all $j$ and thereby, proves the result.
\end{proof}

\color{black}
The previous proof used 
 the compact embedding of $\mathring{H}^{s, p, \d}(\Omega)$ into $L^p(\Omega_\d)$, which is the subject of the following lemma. We point out that it builds substantially on the identification of $\Nzero$ from Theorem~\ref{th:Ncalchar}. 
\begin{lemma}\label{le:comp}
Let $p \in (1,\frac{2}{1-s})$ and suppose $\Omega \subset \R^n$ is a bounded $C^{1,1}$-domain. If $(u_j)_j\subset \mathring{H}^{s,p,\d}(\Omega)$ is such that $u_j\weakly u$ in $H^{s, p, \d}(\Omega)$ as $j\to \infty$ with some $u\in H^{s,p, \delta}(\Omega)$, then $u\in \mathring{H}^{s,p, \delta}(\Omega)$ and
\begin{align}\label{convujstrong}
u_j\to u \quad \text{ in $L^p(\Omega_\d)$ as $j\to \infty$.}
\end{align}
\end{lemma}

\begin{proof}
The fact that $u\in \mathring{H}^{s,p, \delta}(\Omega)$ is clear, since $\mathring{H}^{s,p, \delta}(\Omega)$ is weakly closed in $L^p(\Omega_\d)$.
As for~\eqref{convujstrong}, we use the extension operator $\Ecal_\d^s$ from Section~\ref{subsec:extension} to obtain  \begin{align*}
\Ecal^s_\d u_j \to \Ecal^s_\d u \quad \text{ in $L^p(\Omega_\d)$ as $j\to \infty$}  
\end{align*}
 by~\eqref{conv:extensionmodulo}.  
Therefore, with the sequence $(h_j)_j\subset \Nzero$ given by $h_j := u_j-\Ecal^s_\d u_j$ and $h:=u-\Ecal^s_\d u$ it holds that
\begin{align}\label{conv_hj}
h_j \weakto h \ \text{in $L^p(\Omega_\d)$}\quad \text{ and } \quad h_j \to h  \ \text{ in $L^p(\Gamma_\d)$ as $j\to \infty$, }
\end{align}
where the second convergence follows from $u_j=0=u$ a.e.~on $\Gamma_\d$. If we consider the norm on $\Nzero$ from Remark~\ref{rem:normonN}, then~\eqref{conv_hj} implies
\begin{align*}
\vertiii{h_j}_{\Nzero} = \norm{h_j}_{L^p(\Omega_j)} + \abslr{\int_{\Omega}Q_\delta^s \ast h_j}  \to 0 \quad \text{as $j\to \infty$.}
\end{align*}
Since $\vertiii{\cdot}_{\Nzero}$ is equivalent to the norm induced on $\Nzero$ by $\norm{\cdot}_{L^p(\Omega_\d)}$, we obtain $h_j \to h$ in $L^p(\Omega_\d)$, and thus, 
\[
u_j = \Ecal^s_\d u_j + h_j \to \Ecal^s_\d u +h = u \quad \text{in $L^p(\Omega_\d)$ as $j\to \infty$,}
\]
which concludes the proof.
\end{proof}

\begin{remark}
The contradiction argument in Theorem~\ref{th:poincspecific} works more generally for any weakly closed subset $X \subset H^{s,p,\d}(\Omega)$ that is compactly contained in $L^p(\Omega_\d)$ and satisfies $X\cap \Nzero=\{0\}$. For example, one could replace the condition $\int_{\Omega} u\,dx=0$ by the condition $\int_{\Omega}Q^s_\d *u\,dx=0$ or remove the mean-value condition completely and assume $u=0$ a.e.~in $\Omega_\d \setminus \overline{O}$ for any $O \Subset \Omega$ (cf.~Remark~\ref{rem:enlargedsinglelayer}).
\end{remark}

\subsection{Nonlocal Poincar\'{e}-Wirtinger inequality}\label{sec:poincwirt}
Here, we derive an inequality involving the nonlocal gradient in the spirit of the classical Poincar\'{e}-Wirtinger inequality, by subtracting suitable functions with zero nonlocal gradient. Moreover, we complement the inequality with a compactness result. This will be used later in Section~\ref{sec:neumann} to prove the well-posedness and localization as $s \uparrow 1$ of nonlocal variational problems with Neumann-type boundary conditions.

Let $p \in (1,\infty)$, and consider the metric projection $\pro:L^p(\Omega_\d) \to \Nzero$, 
which minimizes the distance to the functions with vanishing nonlocal gradient in the $L^p$-norm, i.e., for $u\in L^p(\Omega_\delta)$,
\[
\norm{u-\pro(u)}_{L^p(\Omega_{\d})} = \min_{h \in \Nzero} \norm{u-h}_{L^p(\Omega_\d);}
\]
Note that the minimum exists, considering that $\Nzero$ is weakly closed in $\Hspd(\Omega)$, and also in $L^p(\Omega)$, since $\norm{\cdot}_{H^{s, p, \d}(\Omega)}=\norm{\cdot}_{L^p(\Omega)} $ on $\Nzero$. In the case $p=2$, $\pro$ corresponds to the (linear) orthogonal projection onto $\Nzero$. Even though $\pro$ need not be linear when $p \not =2$, one does have that $\pro$ is $1$-homogeneous and that
\begin{align}\label{aux:pides}
\pro(u+h)=\pro(u)+h \qquad \text{for all $h \in \Nzero$.}
\end{align} 
It is also well-known that $\pro$ is continuous, given that $L^p(\Omega_\d)$ is uniformly convex, see~e.g.,~\cite{GoR84}.  

We now formulate and prove the Poincar\'{e}-Wirtinger inequality with the help of the metric projection.
\begin{lemma}[Nonlocal Poincar\'e-Wirtinger inequality]
\label{cor:poincareN}
Let $p \in (1,\infty)$. Then, there exists a constant $C=C(\Omega,p,\d)>0$ such that 
\[
\norm{u-\pro(u)}_{L^p(\Omega_{\d})} \leq C\norm{D^s_\d u}_{L^p(\Omega;\R^n)}
\]
for all $u \in H^{s,p,\d}(\Omega)$.
\end{lemma}

\begin{proof}
It follows from Theorem~\ref{th:connecquo} (in the version where $\widetilde{H}^{s,p,\delta}(\Omega)$ is equipped with the quotient norm, see Remark~\ref{rem:quotientnorm}) and Lemma~\ref{le:translation1} that 
\begin{align*}
\norm{u-\pro(u)}_{L^p(\Omega_\d)} &\leq \vertiii{[u]_\delta^s}_{\widetilde{H}^{s,p,\d}(\Omega)} =\vertiii{\widetilde{\Pcal}^s_\d\widetilde{\Qcal}^s_\d[u]^s_\d}_{\widetilde{H}^{s,p,\d}(\Omega)}\leq C\norm{\widetilde{\Qcal}_\d^s[u]_\delta^s}_{\widetilde{W}^{1,p}(\Omega)}  \\ & = C\norm{[\Qcal_\delta^s u]}_{\widetilde{W}^{1,p}(\Omega)}  = C\norm{\nabla (\Qcal_\delta^s u)}_{L^p(\Omega;\R^n)}=C\norm{D^s_\d u}_{L^p(\Omega;\R^n)},
\end{align*}
with a constant $C>0$ independent of $s$.  
\end{proof}
Second, one obtains the following compactness result. It can be seen as the trace-free analogue to~\cite[Theorem~2.3]{BCM20} in the setting of complementary-value spaces.
\begin{lemma}[Compactness in $H^{s,p, \d}(\Omega)$]\label{le:rellichprojection}
 Let $p \in (1,\infty)$, then any sequence $(u_j)_j\subset H^{s, p, \d}(\Omega)$ converging weakly to $u$ in $H^{s, p, \d}(\Omega)$ satisfies
\begin{center}
$u_j-\pro(u_j) \to u-\pro(u)$ \quad in $L^p(\Omega_{\d})$ as $j\to \infty$. 
\end{center}
\end{lemma}

\begin{proof}
Using the extension operator modulo $\Nzero$ from~\eqref{extensionmodulo}, we define 
\begin{center}
$h_j:=\Ecal^s_\d u_j-u_j+\Ecal^s_\d u-u \in \Nzero$\quad for all $j$. 
\end{center}
Since $\Ecal^s_\d u_j \to \Ecal^s_\d u$ in $L^p(\Omega_\d)$ according to~\eqref{conv:extensionmodulo}, it follows that $u_j+h_j \to u$ in $L^p(\Omega_\d)$, and hence,
\[
\lim_{j \to \infty}\norm{u_j-u-\pro(u_j-u)}_{L^p(\Omega_\d)} \leq \lim_{j \to \infty} \norm{u_j-u+h_j}_{L^p(\Omega_\d)}=0,
\]
by definition of the metric projection. 
This shows that
\begin{center}
 $u_j-\pro(u_j-u)\to u$ in $L^p(\Omega_\delta)$ as $j\to \infty$.
 \end{center} In view of~\eqref{aux:pides} and the continuity of $\pro$, we then find that
\[
u_j -\pro(u_j)  
= u_j -\pro(u_j-u)  -\pro(u_j -\pro(u_j-u)) \to u- \pro(u) \quad\text{in $L^p(\Omega_\d)$},
\]
which concludes the proof.
\end{proof}

\section{Nonlocal differential inclusion problems} \label{sec:diffinclusions}

In the present section we discuss results on the solvability of differential inclusion problems involving the nonlocal gradient. This means that for a given set $E\subset \R^{m\times n}$, we aim to find all $u\in H^{s, p, \delta}(\Omega;\R^m)$ that satisfy
\begin{align}\label{nonlocinclusion_motivation}
D^s_\d u  \in E \quad \text{a.e.~ in $\Omega$},
\end{align}
and optionally, also a boundary condition in the single layer $\Gamma_\d$ or the double layer $\Gamma_{\pm\d}$. Problems of the type \eqref{nonlocinclusion_motivation} have not appeared in the literature before, although related results such as fractional Korn inequalities have been studied recently in various settings \cite{BCM23Eringen, ScM19, HMM23}.

Throughout this section, let $p\in [1, \infty]$ and $\Omega \subset \R^n$ be a bounded $C^{1,1}$-domain. Additionally, whenever we work with Dirichlet conditions in the double layer $\Gamma_{\pm \d}$, we also assume that $\Omega_{-\d} \not= \emptyset$ and $\abs{\partial \Omega_{-\d}}=0$. The set $\Nzero$ naturally plays a key role in the discussion of \eqref{nonlocinclusion_motivation}, considering that it can be interpreted as the solution to the most basic nonlocal inclusion, namely,  with the choice $E=\{0\}$. On the one hand, for any solution to~\eqref{nonlocinclusion_motivation}, adding a function from $\Nzero$ generates a new solution, that is, if $u\in \Hspd(\Omega)$ solves~\eqref{nonlocinclusion_motivation}, then so does any other element in $[u]_\delta^s = u+\Nzero$, cf.~Section~\ref{sec:app1}. When $p \in (1,\frac{2}{1-s})$, any single-layer boundary condition $g \in L^p(\Gamma_\d)$ can therefore be attained, by the characterization of $\Nzero$ in Theorem~\ref{th:Ncalchar}.

Our overall strategy in dealing with~\eqref{nonlocinclusion_motivation} is to relate them with classical differential inclusions, and to carry over the by now well-known results on their classical counterparts, that is, solving 
\begin{equation}\label{eq:classicalinclusion}
\nabla v \in E \quad \text{a.e. in $\Omega$.}
\end{equation}
for $v\in W^{1,p}(\Omega)$, also subject to boundary conditions. A rich literature on the latter has emerged over the last decades, including~\cite{DaM99, DaM99b, DaP05, MuS01, Syc11}, see also~\cite{Dacorogna, Mul99, Rin18} for an overview.
While there is no unified theory available, the results fall roughly into two groups, relating to the complementary themes of rigidity and flexibilty. This division, which we will adopt here as well, is partly motivated by models in materials science, 
where differential inclusions appear naturally when studying microstructure formation, cf.~\cite{Mul99, Rin18}.

The connection between nonlocal and standard gradients established in Section~\ref{sec:app1} implies that~\eqref{nonlocinclusion_motivation} and \eqref{eq:classicalinclusion} are equivalent when it comes to solvability. 
Indeed, due to Theorem~\ref{th:connecquo} the map $\widetilde{\Qcal}_\d^s$ gives a bijection between the solutions of \eqref{eq:classicalinclusion} modulo constants and the solutions to \eqref{nonlocinclusion_motivation} modulo functions in $\Nzero$. In the following, we take a look into selected aspects of flexibility and rigidity in the nonlocal setting, starting with the latter.

 One calls the classical differential inclusion~\eqref{eq:classicalinclusion} rigid, if all its solutions $v$ have constant gradient, meaning that, $v(x)=l_A(x) +c=Ax+c$ for $A \in E$ and $c \in \R^m$; recall the notation $l_A$ with $A\in \Rmn$ for the linear function $l_A(x)=Ax$ with $x\in \R^n$.
The nonlocal gradient of a linear function agrees with the classical gradient, since 
\begin{align}\label{gradaffine}
D^s_\d l_A= Q^s_\d * \nabla l_A = Q^s_\d * A = A,
\end{align}
where we have used $\norm{Q^s_\d}_{L^1(\R^n)}=1$ (see also~\cite[Proposition~4.1]{BCM23Eringen}). 
Based on this observation, one obtains that rigidity carries over to the nonlocal setting in the following sense.  
\color{red} 

\color{black}
\begin{corollary}[Nonlocal rigidity]\label{cor:rigidsingle}
Let 
$E \subset \Rmn$ be such that the differential inclusion \eqref{eq:classicalinclusion} is rigid. Then, all solutions $u \in H^{s,p,\d}(\Omega)$ to the nonlocal inclusion
\begin{align}\label{eq:nonlocalinclusion}
D^s_\d u \in E \quad \text{a.e.~in $\Omega$},
\end{align}
are of the form $u=l_A+h$ with $A\in E$ and $h \in \Nzero$.
In particular, if $p \in (1,\frac{2}{1-s})$, then for any $g \in L^p(\Gamma_\d)$, there is a solution $u$ of \eqref{eq:nonlocalinclusion} with $u=g$ a.e.~in $\Gamma_\d$.
\end{corollary}

\begin{proof}
As $u\in [l_A]_\delta^s$ with $A\in E$ clearly solves~\eqref{eq:nonlocalinclusion} in view of~\eqref{gradaffine}, it remains to show that these are the only solutions. 
Indeed, by the assumption of rigidity, the solutions to \eqref{eq:classicalinclusion} are exactly the functions that lie in $[l_A]$ for some $A \in E$, so that any $u$ solving~\eqref{eq:nonlocalinclusion} needs to satisfy $\widetilde{\Qcal}_\delta^s[u]_\delta^s= [l_A]$.  Since also $\widetilde{\Qcal}_\delta^s[l_A]_\delta^s= [l_A]$ and $\widetilde{\Qcal}_\d^s$ is injective according to Theorem~\ref{th:connecquo}, we finally conclude that $u-l_A \in \Nzero$.

\color{black}
When the assumptions of Theorem~\ref{th:Ncalchar} are satisfied, 
we may use Theorem~\ref{th:Ncalchar} to find that any boundary condition is attained in $\Nzero$, which yields the second part of the statement.
\end{proof}

The preceding result characterizes all solutions in terms of the set $\Nzero$ and shows that there is no restriction on the boundary conditions that can be achieved in the single layer. If one prescribes boundary conditions in the double layer $\Gamma_{\pm \d}$, instead, the set of solutions is considerably more restrictive. 
Our next statement addresses a nonlocal inclusion problem with linear boundary data $l_A$ with $A\in \Rmn$, precisely, 
\begin{align}\label{nonlocinclusion_with}
\begin{cases}
D^s_\d u \in E &\text{a.e.~in $\Omega_{-\d}$},\\
D^s_\d u =A & \text{a.e.~in $\Gamma_{-\d}$},\\
u = l_A &\text{a.e.~in $\Gamma_{\pm\d}$},
\end{cases}
\end{align}
for $u\in \Hspd(\Omega)$. Note that the reason for prescribing the nonlocal gradient in the collar $\Gamma_{-\d}$ is that the condition $u=l_A$ a.e.~in $\Gamma_{\pm \d}$ automatically implies $D^s_\d u =A$ near $\partial \Omega$  in light of \ref{itm:h2}.  The inclusion $D^s_\d u \in E$ a.e.~in $\Omega$ would therefore only be possible if $A \in E$, which renders the problem trivial.  We now show a rigidity statement for \eqref{nonlocinclusion_with}. 

\begin{corollary}[Nonlocal rigidity with linear boundary conditions]\label{cor:rigiddouble}
Let  $E \subset \Rmn$ be such that the inclusion \eqref{eq:classicalinclusion} is rigid and let $A\in \Rmn$. Then, the nonlocal inclusion problem~\eqref{nonlocinclusion_with} has a solution if and only if $A \in E$, which is then uniquely given by $u=l_A$.
\end{corollary}

\begin{proof}
Let $u$ be a solution of \eqref{nonlocinclusion_with} and define $v:=\Qcal^s_\d u \in W^{1,p}(\O)$, which then satisfies $\nabla v \in E$ a.e.~in $\Omega_{-\d}$ and $\nabla v =A$ a.e.~in $\Gamma_{-\d}$, cf.~Lemma~\ref{le:translation1}. Since \eqref{eq:classicalinclusion} is rigid, there is an $A' \in E$ such that $\nabla v=A'$ a.e.~in $\Omega_{-\d}$. Hence, it holds that $v=l_{A'}+c$ a.e.~in $\Omega_{-\d}$ for some $c \in \R$. Moreover, for a.e.~$x \in \Omega \setminus (\overline{\Omega_{-\d}} + \supp(Q^s_\d))$  (by~\ref{itm:h2}, this open set is non-empty),  we obtain 
\[
v(x) =(Q^s_\d * u)(x)  = (Q_\delta^s \ast l_A)(x)= Ax.
\]
Combining this with $\nabla v =A$ a.e.~in $\Gamma_{-\d}$, yields that $v=l_A$ a.e.~in $\Gamma_{-\d}$. We conclude that
\[
\begin{cases}
v = l_{A'}+c &\text{a.e.~in $\Omega_{-\d}$},\\
v = l_{A} &\text{a.e.~in $\Gamma_{-\d}$},
\end{cases}
\]
so that we must have $l_{A'}+c=l_{A}$ on $\partial \Omega_{-\d}$ for $v$ to be a Sobolev function. Unless $A=A'$ and $c=0$, we find that the set where $l_{A'}+c=l_{A}$ is an affine subspace of dimension at most $n-1$, which cannot contain the boundary of the bounded open set $\Omega_{-\d}$. Therefore, we must have $A=A'$ and $c=0$, which yields, in particular, that $A \in E$ and $D^s_\d u =\nabla v =A$ a.e.~in $\Omega$. We now infer from the nonlocal Poincar\'{e} inequality for double-layer boundary conditions (see~\cite[Theorem~6.1]{BCM23}) that $u=l_A$ is indeed the only solution.
\end{proof}

Next is a statement on flexibility for~\eqref{nonlocinclusion_with}, which also allows for solutions with non-constant nonlocal gradients and reveals a relation between the attainable boundary conditions and the set $E$. In doing so, we restrict our attention to a weaker notion of solutions, though, calling a sequence  $(u_j)_j \subset H^{s,\infty,\d}(\Omega;\R^m)$ an approximate solution to~\eqref{nonlocinclusion_with}, if
\begin{equation}\label{eq:nonrigidinclusion}
\begin{cases}
\dist(D^s_\d u_j,E) \to 0 & \text{in measure on $\Omega_{-\d}$,}\\
D^s_\d u_j \to A & \text{in measure on $\Gamma_{-\d}$,}\\
u_j = l_A & \text{in $\Gamma_{\pm\d}$}.
\end{cases}
\end{equation}

In the classical case, it is well-known that approximate solutions to~\eqref{eq:classicalinclusion} subject to linear boundary values $l_A$ exist if and only if $A$ lies in the quasiconvex hull of $E$ defined by
\[
E^{\rm qc}:=\bigl\{B \in \Rmn\,:\, f(B) \leq \sup_{E} f \ \text{for all quasiconvex $f:\Rmn \to \R$}\bigr\},
\]
see~e.g.,~\cite[Theorem~4.10]{Mul99},\cite[Chapter~7]{Dacorogna}. For the approximate solutions as in \eqref{eq:nonrigidinclusion}, we can use the translation method to prove an analogous statement. 

\color{black}
\begin{proposition}[Approximate solutions to nonlocal differential inclusions]
Let $E \subset \Rmn$ be compact and $A \in \Rmn$. Then, \eqref{nonlocinclusion_with} admits an approximate solution in the sense of \eqref{eq:nonrigidinclusion} if and only if $A \in E^{\rm qc}$.
\end{proposition}

\begin{proof}
First, suppose that $A \in E^{\rm qc}$, then by \cite[Theorem~4.10]{Mul99}, there is a bounded sequence $(v_j)_j \subset W^{1,\infty}_0(\Omega;\R^m)$ such that 
\begin{align}\label{aux83}
\dist(A+\nabla v_j,E) \to 0 \quad \text{in measure on $\Omega$. }
\end{align}
We may assume without loss of generality that $v_j \to 0$ in $L^{\infty}(\Omega;\R^m)$ and hence, also $\nabla v_j \starto 0$ in $L^{\infty}(\Omega;\Rmn)$; otherwise, we glue together suitably scaled and translated copies of $v_j$ for each $j$. 

  After identifying $v_j$ with its extension to $\R^n$ by zero, we define the sequence $(\tilde u_j)_j$ by 
\begin{align*}
\tilde u_j := \Pcal^s_\d v_j \in H^{s,\infty,\d}(\R^n;\R^m) \quad \text{for $j \in \N$}.
\end{align*} 
Since $D^s_\d \tilde u_j=\nabla v_j \starto 0$ in $L^{\infty}(\R^n;\Rmn)$, and the sequence $(v_j)_j$ is also bounded in $W^{1,p}(\R^n;\R^m)$, it follows along with the weak continuity of $\Pcal_\d^s$ that
$\tilde u_j \weakto 0$
 in $H^{s,p,\d}(\R^n;\R^m)=H^{s,p}(\R^n;\R^m)$ as $j\to \infty$ for all $p \in (1,\infty)$.  
 In addition, the compact embedding of $H^{s,p}(\R^n;\R^m)$ into $L^{\infty}(\Omega_\d;\R^m)$ for $sp >n$ (see~Section~\ref{sec:sobolevspaces}), yields
 \begin{align}\label{aux82}
 \tilde u_j \to 0 \quad\text{ in $L^{\infty}(\Omega_\d;\R^m)$\quad  and \quad $\tilde u_j \starto 0$ in $H^{s,\infty,\d}(\Omega;\R^m)$. }
 \end{align}

We now introduce a sequence of cut-off functions $(\chi_j)_j \subset C_c^{\infty}(\Omega_{-\d};[0,1])$ such that 
\begin{align}\label{aux84}
\abs{\Omega_{-\d} \setminus \{\chi_j=1\}} \to 0 \quad\text{ and}\quad 
\mathrm{Lip}(\chi_j)\norm{\tilde u_j}_{L^{\infty}(\Omega_\d;\R^m)} \to 0  \quad \text{ as $j \to \infty$,} 
\end{align}
where ${\mathrm{Lip}}(\chi_j)$ denotes the Lipschitz constant of $\chi_j$, 
and we define $(u_j)_j$ via
\begin{align*}
u_j:=\chi_j \tilde u_j \in H^{s,\infty,\d}(\Omega;\R^m), 
\end{align*}
which guarantees 
\begin{align}\label{aux81}
u_j =0 \quad \text{ in $\Gamma_{\pm\d}$.}
\end{align}
 Moreover, by the nonlocal Leibniz rule (see \cite[Lemma~2]{CKS23}),
\begin{align}\label{aux85}
D^s_\d u_j = \chi_j D^s_\d \tilde u_j + K_{\chi_j}(\tilde u_j) = \chi_j \nabla v_j +K_{\chi_j}(\tilde u_j),
\end{align}
where $K_{\chi_j}:L^\infty(\Omega_\delta)\to L^\infty(\Omega;\R^n)$ are bounded linear operators that satisfy 
\begin{align}\label{aux86}
\norm{K_{\chi_j}(\tilde u_j)}_{L^{\infty}(\Omega;\Rmn)} \leq C\mathrm{Lip}(\chi_j)\norm{\tilde u_j}_{L^{\infty}(\Omega_\d;\R^m)} \to 0 \ \text{ as $j \to \infty$};
\end{align}
the last convergence follows from~\eqref{aux84}.
Since $\dist(A+\chi_j \nabla v_j,E) \to 0$ in measure on $\Omega_{-\d}$ due to~\eqref{aux83} and the first convergence in~\eqref{aux84}, we conclude along with~\eqref{aux85} and~\eqref{aux86} that 
\begin{align*}
\dist(A+D^s_\d u_j,E) \to 0 \quad \text{ in measure on $\Omega_{-\d}$. }
\end{align*} 
Moreover, as \cite[Lemma~3]{CKS23} yields convergence $D_\delta^s u_j \to 0$ in $L^{\infty}$ in any compactly contained subset of the collar $\Gamma_{-\d}$, we have that 
\[
A+D^s_\d u_j \to A \quad \text{in measure on $\Gamma_{-\d}$ as $j \to \infty$}.
\]
Hence, we obtain the desired approximate solution to \eqref{nonlocinclusion_with}, after adding the linear function $l_A$ to $(u_j)_j$. \smallskip

\color{black}
Conversely, if $(u_j)_j \subset H^{s,\infty,\d}(\Omega;\R^m)$ is a sequence satisfying \eqref{eq:nonrigidinclusion}, we set $v_j:=\Qcal^s_\d u_j$ for all $j\in \N$ to find that $(v_j)_j \subset W^{1,\infty}(\Omega;\R^m)$ is a sequence with $v_j=l_{A}$ on $\partial \Omega$ for all $j \in \N$ in the sense of traces, and
\[
\begin{cases}
\dist(\nabla v_j, E)=\dist(D^s_\d u_j,E) \to 0 & \text{in measure on $\Omega_{-\d}$ as $j \to \infty$},\\
\nabla v_j=D^s_\d u_j \to A &\text{in measure on $\Gamma_{-\d}$ as $j \to \infty$}.
\end{cases}
\]
A small adaptation to the argument in \cite[Theorem~4.10\,(i)]{Mul99} now shows that $A \in E^{\rm qc}$, as desired. 
\end{proof}
\section{Well-posedness and localization of nonlocal Neumann-type problems}\label{sec:neumann}

 This section is concerned with the analysis of nonlocal differential equations with homogeneous Neumann-type boundary conditions. In fact, it even covers a more general setting with natural boundary conditions.
Our main results are the well-posedness for these problems for any fixed fractional parameter $s \in (0,1)$ 
 and a rigorous proof of localisation, i.e., the convergence to the classical analogues of these boundary-value problems as the fractional parameter $s$ goes to $1$.

We approach these problems from the variational perspective, where the objects of interest are the associated energy functionals: For $\Omega\subset \R^n$ a bounded Lipschitz domain and $p\in (1, \infty)$, consider
$\Fcal_\d^s:H^{s,p, \d}(\Omega;\R^m)\to \R_\infty$ given by
\begin{align}\label{functional_Fds}
\Fcal_\d^s(u)=\int_{\Omega} f(x,D^s_\d u)\,dx-\int_{\Omega_\d}F \cdot u\,dx,
\end{align}\color{black}
where $F \in L^{p'}(\Omega_\d;\R^m)$ with $p'$ the dual exponent of $p$ and the Carath\'{e}odory function $f:\Omega \times \Rmn \to \R_{\infty}$ are suitably given.

Due to the absence of any constraints in the space of admissible functions $\Hspd(\Omega)$, the minimization of $\Fcal_\delta^s$  gives rise to natural boundary conditions when passing to the Euler-Lagrange equations. Nonlocal variational problems on complementary-value spaces, in contrast, lead to Dirichlet boundary-value problems, see  e.g.,~\cite[Section~8]{BCM23}.

\smallskip

\color{black}

\subsection{Existence theory for a class of nonlocal Neumann-type variational problems} \label{sec:wellposed}

In this section we prove the existence of minimizers of the functional in \eqref{functional_Fds}, on a suitable subspace of $H^{s,p,\d}(\O)$ where the Poincar\'{e}-Wirtinger inequality from Section~\ref{sec:poincwirt} can be applied. Precisely, recalling the metric projection $\pi^s_\d:L^p(\O_\d;\R^m) \to \Nzerom$ from Section~\ref{sec:poincwirt} (extended to vector-valued functions), we introduce the sets 
\[
{\Nzerom}^{\perp} = \{u \in H^{s,p,\d}(\Omega;\R^m)\,:\, \pro(u)=0\}.
\]
For $p=2$, this corresponds to the orthogonal complement of $N^{s,2,\d}(\O;\R^m)$ in $L^2$, whereas for the case $p\not =2$, it need not be a linear subspace, given the nonlinearity of the metric projection.

We now present the main result of this section, which establishes the existence of minimizers for $\Fcal^s_\d$ on the subspaces ${\Nzerom}^{\perp}$. 

\begin{theorem}[Existence of minimizers for $\Fcal_\d^s$]
\label{th:wellposed} 
Let $p \in (1,\infty)$, $F \in L^{p'}(\Omega_\d;\R^m)$ and $f:\Omega \times \Rmn \to \R_{\infty}:=\R \cup \{\infty\}$ be a Carath\'{e}odory integrand such that 
\[
f(x,A) \geq c\,\bigl(\abs{A}^p-1\bigr) \quad \text{for a.e.~$x \in \Omega$ and all $A \in \Rmn$}
\]
with a constant $c>0$. If $v \mapsto \int_\Omega f(x,\nabla v)\,dx$ is weakly lower semicontinuous on $W^{1,p}(\Omega;\R^m)$, then the functional $\Fcal_\d^s$ in~\eqref{functional_Fds}, i.e.,
\begin{align*}
\Fcal_\delta^s(u) = \int_{\Omega} f(x,D^s_\d u)\,dx-\int_{\Omega_\d}F \cdot u\,dx, 
\end{align*}
admits a minimizer over ${\Nzerom}^{\perp}$.
\end{theorem}

\begin{proof}
We apply the direct method in the calculus of variations. Note first that ${\Nzerom}^\perp$ is a weakly closed subset of $H^{s, p, \delta}(\Omega;\R^m)$ as a consequence of Lemma~\ref{le:rellichprojection}. The coercivity then follows from the lower bound on $f$ along with the Poincar\'e-Wirtinger inequality from Corollary~\ref{cor:poincareN}, which reduces to
\[
\norm{u}_{L^p(\Omega_\d;\R^m)} \leq C \norm{D^s_\d u}_{L^p(\Omega;\Rmn)},
\]
for $u \in {\Nzerom}^{\perp}$. 
For the weak lower semicontinuity of $\Fcal_\delta^s$, 
we observe that if $u_j \weakto u$ in $H^{s,p,\d}(\Omega;\R^m)$, then $v_j:=\Qcal^s_\d u_j \weakto \Qcal^s_\d u=:v$ in $W^{1,p}(\Omega;\R^m)$ with $\nabla v_j = D^s_\d u_j$ for all $j$ and $\nabla v = D^s_\d u$, cf.~Lemma~\ref{le:translation1}. Hence,
\begin{align*}
\Fcal_\delta^s(u)&=\int_{\Omega} f(x,\nabla v)\,dx - \int_{\Omega_\d} F \cdot u\,dx\\
&\leq \liminf_{j \to \infty} \int_{\Omega} f(x,\nabla v_j)\,dx - \int_{\Omega_\d} F \cdot u_j\,dx = \liminf_{j \to \infty}\Fcal_\delta^s(u_j),
\end{align*}
showing that $\Fcal_\delta^s$ is weakly lower semicontinuous on $H^{s, p, \d}(\Omega;\R^m)$. 
In combination with the coercivity, this yields the desired existence of a minimizer of $\Fcal_\delta^s$ in ${\Nzerom}^\perp$. 
\end{proof}

\begin{remark}\label{rem:wellposed}
a) For the sake of generality, the previous theorem assumes that the classical integral functional (with standard gradients) associated to $\Fcal_\delta^s$ is weakly lower semicontinuous. Well-known sufficient conditions for this include polyconvexity of the integrand $f$ in the second argument or quasiconvexity of the latter 
 along with a suitable upper bound, see e.g.,~\cite[Theorems~8.11 and 8.31]{Dacorogna}.\smallskip 
 
b) Note that if $F\in L^{p'}(\Omega_\d;\R^m)$ satisfies the compatibility condition
\begin{align}\label{conditioncompatibility}
\int_{\Omega_\d} F \cdot h\,dx = 0 \quad \text{for all $h \in \Nzerom$},
\end{align}
then $\Fcal_\delta^s$ is invariant under translations in $\Nzerom$. As a consequence of Theorem~\ref{th:wellposed}, $\Fcal_\delta^s$ then admits minimizers over the whole space $H^{s,p,\d}(\Omega;\R^m)$.
\end{remark}

As a consequence of Theorem~\ref{rem:wellposed}, and specifically Remark~\ref{rem:wellposed}, one can infer, by passing to Euler-Lagrange equations, the existence of  weak solutions for a class of nonlocal differential equations with natural boundary conditions. Namely, suppose that $F$ satisfies~\eqref{conditioncompatibility} and let $f$ be continuously differentiable in its second argument and $C>0$ such that
\begin{align}\label{eq:growth}
\abs{f(x,A)} \leq C(1+\abs{A}^p) \quad \text{and} \quad \abs{D_A f(x,A)} \leq C(1+\abs{A}^{p-1}) \quad \text{for all $(x,A) \in \Omega_\d \times \Rmn$},
\end{align}
with $D_A f$ the differential of $f$ with respect to its second argument. Then, using a standard argument, see~\cite[Theorem~3.37]{Dacorogna}, we find that the minimizers $u \in H^{s,p,\d}(\Omega;\R^m)$ of $\Fcal_\d^s$ solve the  weak Euler-Lagrange equation
\begin{align}\label{eq:weakel}
\int_{\Omega} D_A f(x,D^s_\d u) \cdot D^s_\d v\,dx =\int_{\Omega_\d} F \cdot v\,dx \quad \text{for all $ v \in H^{s,p,\d}(\Omega;\R^m)$}.
\end{align}
We note that the compatibility condition in \eqref{conditioncompatibility} is also necessary for \eqref{eq:weakel} to hold, since the left hand-side is zero for $v \in \Nzerom$. Moreover, by the definition of the weak nonlocal divergence via nonlocal integration by parts, the equation \eqref{eq:weakel} corresponds to the weak formulation of
\begin{align}\label{ELequation}
-\Div^s_\d \bigl( \mathbbm{1}_{\Omega} D_A f(\cdot,D^s_\d u)\bigr) = F \quad \text{in $\Omega_\d$}. 
\end{align}

Within the region $\Omega_{-\d}$, this equation 
reduces to the nonlocal Euler-Lagrange equation from \cite[Theorem~8.2]{BCM23}, while in the double boundary layer $\Gamma_{\pm\d}$, the equation takes into account the geometry of the boundary $\partial \Omega$. More precisely, one obtains
\begin{align}\label{ELequation_split}
\begin{cases}
-\Div^s_\d (D_A f(x,D^s_\d u)) = F \quad  &\text{in $\Omega_{-\d}$},\\
\Ncal_{\d}^s (D_Af(x, D_\delta^s u)) =F \quad & \text{in $\Gamma_{\pm \d}$,}
\end{cases}
\end{align}
where $\Ncal_{\d}^{s}:=-\Div^s_\d (\mathbbm{1}_{\Omega} \,\cdot\,)$ coincides with the nonlocal boundary operator,  recently introduced in~\cite[Definition~3.1]{BCFR23} to prove a concise nonlocal integration by parts formula.

Now, if $u$ solves~\eqref{ELequation} or~\eqref{ELequation_split} (weakly), the nonlocal divergence imposes that $\mathbbm{1}_{\Omega} D_A f(\cdot,D^s_\d u)$ must be regular enough across $\partial \Omega$. As $s \uparrow 1$, we expect to recover the natural boundary conditions $D_A f(\cdot,\nabla u)\cdot \nu=0$ on $\partial \Omega$ with $\nu$ an outer normal to $\partial\Omega$. This intuition is made rigorous in the next section.

\subsection{Localization for $s\uparrow 1$}\label{sec:localisation}

We now turn to studying the limiting behavior of the nonlocal variational problem from Theorem~\ref{th:wellposed}, and the closely related nonlocal Neumann-type problems, as the fractional parameter $s$ tends to $1$. Our main result in this section (see Theorem~\ref{th:local}) rigorously confirms the expectation that these problems localize, that is, they converge to their classical counterparts with usual gradients. 

To start, let us collect in the next lemma some preparatory tools revolving around the asymptotic behavior of the sets $\Nzero$ and $N^{s, p,\delta}(\Omega)^\perp$ as $s$ tends to $1$. To capture the limit objects, we introduce $H^{1,p,\d}(\Omega):=\{ u \in L^p(\Omega_\d)\,:\, u|_{\Omega} \in W^{1,p}(\Omega)\}$ and 
\begin{align}\label{N1}
N^{1, p, \delta}(\Omega) = \{u \in H^{1,p,\d}(\Omega)\,:\, \nabla u = 0 \ \text{in $\Omega$}\} = \{u\in L^p(\Omega_\d): u|_{\Omega} \ \text{is constant}\}.
\end{align}
along with its corresponding metric projection $\pi^{1}_\delta:L^p(\Omega_\delta)\to N^{1, p, \delta}(\Omega)$,  
and we also set 
\[
{N^{1, p, \delta}(\Omega)}^{\perp} := \{u \in H^{1,p,\d}(\Omega;\R^m)\,:\, \pi_\delta^1(u)=0\}.
\]
Given the definition~in~\eqref{N1}, the projection $\pi_\delta^1(u)$ agrees with $u$ in $\Gamma_\d$ and is constant on $\Omega$. Considering that $\argmin_{c \in \R} \norm{u-c}_{L^p(\Omega)} =0$ is equivalent to $\int_{\Omega} |u|^{p-1}\sign(u)\, dx=0$ for any $u\in L^p(\Omega_\delta)$, 
one can represent $N^{1, p, \delta}(\Omega)^\perp$ as
\begin{align}\label{N1perp}
N^{1,p, \delta}(\Omega)^{\perp} 
 =\left\{u\in L^p(\Omega_\delta): u|_\Omega\in W^{1,p}(\Omega), \ u=0 \text{ a.e. in $\Gamma_\d$,  $\int_\Omega |u|^{p-1}\sign(u)\,  dx=0$}\right\}. 
\end{align}
When $p=2$, the nonlinear integral condition in~\eqref{N1perp} reduces simply to the requirement of zero mean value.  
\color{black}
\begin{lemma}\label{le:asymptotics}
Let $p \in (1,\infty)$ and let $(s_j)_j \subset (0,1)$ be a sequence converging to $1$. Then, these statements hold:
\begin{itemize}
\item[$(i)$] For all $v\in W^{1,p}(\R^n)$ it holds that $\Pcal^{s_j}_\d v \to v$ in $L^p(\Omega_\d)$ as $j\to \infty$. \\[-1.5ex]

\item[$(ii)$] If $(u_j)_{j} \subset L^p(\Omega_\d)$ converges to $u \in L^p(\Omega_\d)$, then $\pi_{\delta}^{s_j}(u_j) \to \pi_\delta^1(u)$ as $j \to \infty$. \\[-1.5ex]

\item[$(iii)$] Let $(u_j)_{j} \subset L^p(\Omega_\d)$ with $u_j\in N^{s_j, p, \delta}(\Omega)^\perp$ for all $j$. If $\sup_{j }\norm{D^{s_j}_\d u_j}_{L^p(\Omega;\R^n)} < \infty$, then there is a $u \in N^{1,p,\delta}(\Omega)^\perp$ such that (up to a non-relabeled subsequence)
\[
u_j \to u\ \text{in $L^p(\Omega_\d)$} \quad \text{and} \quad D^{s_j}_\d u_j \weakto \nabla u \ \text{in $L^p(\Omega;\R^n)$ as $j \to \infty$.}
\]
\end{itemize}
\end{lemma}

\begin{proof}
\textit{ Part $(i)$.} Let $v\in W^{1, p}(\R^n)$. In light of \eqref{eq:increasingnorm} and \eqref{eq:uniformpbound}, we find for $0< \bar{s} \leq \inf_{j} s_j$ that
\[
\sup_{j} \norm{\Pcal^{s_j}_\d v}_{H^{\bar{s},p}(\R^n)} \leq \sup_{j } \norm{\Pcal^{s_j}_\d v}_{H^{s_j,p}(\R^n)} < \infty.
\]
Due to the compact embedding of $H^{\bar{s},p}(\R^n)$ into $L^p(\Omega_{2\d})$ (see~Section~\ref{sec:sobolevspaces}), there is a subsequence (not relabeled) such that $\Pcal^{s_j}_\d v \to w$ in $L^p(\Omega_{2\d})$ for some $w\in L^p(\Omega_{2\delta})$.  To identify $w$, consider an arbitrary test function $\varphi \in C_c^{\infty}(\Omega_\d)$. As shown in~\cite[Eq.~(3.4)]{CKS23}, it holds that $\Qcal^{s_j}_\d \varphi = Q_\delta^{s_j}*\varphi \to \varphi$ uniformly as $j \to \infty$. Together with Fubini's theorem, this implies
\begin{align*}
\int_{\Omega_\d} w\varphi\,dx & = \lim_{j \to \infty} \int_{\Omega_{2\d}} (\Pcal^{s_j}_\d v) \,(Q^{s_j}_\d* \varphi)\,dx = \lim_{j \to \infty} \int_{\Omega_\d} [Q^{s_j}_\d * (\Pcal^{s_j}_\d v)]\,\varphi\,dx \\ &= \lim_{j \to \infty} \int_{\Omega_\d} (\Qcal^{s_j}_\d\Pcal^{s_j}_\d v)\,\varphi\,dx  = \int_{\Omega_\d} v\,\varphi\,dx,
\end{align*}
from which we infer $w=v$ on $\Omega_\d$. \smallskip

\textit{Part $(ii)$.} Since $0 \in \Nzero$ for all $s \in (0,1]$, we deduce from the definition of the metric projection that \begin{center}
$\norm{\pi^{s_j}_\delta(u_j)}_{L^p(\Omega_\d)} \leq 2\norm{u_j}_{L^p(\Omega_\d)}$ \quad for all $j$.
\end{center}  
As $(u_j)_j$ is bounded in $L^p(\Omega_\delta)$, so is $(\pi_\delta^{s_j}(u_j))_j$, and there exists a (non-relabeled) subsequence and a $w\in L^p(\Omega_\d)$ with $\pi_\delta^{s_j}(u_j) \weakto w$ in $L^p(\Omega_\d)$ as $j\to \infty$. For any test function $\psi \in C_c^{\infty}(\Omega;\R^n)$, one then obtains 
\begin{align}\label{eq:divw0}
\int_{\Omega} w \Div \psi\,dx = \lim_{j \to \infty} \int_{\Omega_\d} \pi_\delta^{s_j}(u_j)\, \Div^{s_j}_\d \psi\,dx = 0,
\end{align}
where the first inequality uses $\Div^{s_j}_\d \psi \to \Div \psi$ uniformly on $\Omega_\d$ (see~\cite[Lemma~7]{CKS23}), and the last equality follows from integration by parts and the fact that $\pi_\delta^{s_j}(u_j) \in N^{s_j, p, \delta}(\Omega)$ has zero gradient $D_\delta^{s_j}$ for each $j$. By~\eqref{eq:divw0}, the limit function $w$ is constant on $\Omega$, and hence, $w \in N^{1,p, \delta}(\Omega)$, cf.~\eqref{N1}. It remains to show that $w=\pi_\delta^1(u)$ and that $\pi_\delta^{s_j}(u_j)$ converges even strongly. 

To this aim, we first construct an auxiliary sequence $(h_j)_j\subset L^p(\Omega_\delta)$ with the properties that 
\begin{align}\label{construct_hj}
h_j \in N^{s_j, p, \delta}(\Omega) \text{\ \  for all $j$ \qquad and\qquad}  h_j \to \pi_\delta^1(u) \  \text{ in $L^p(\Omega_\d)$ as $j\to \infty$.}
\end{align} 
Since $\pi_\delta^1(u)$ is constant on $\Omega$, one can find a sequence $(\phi_k)_k \subset C_c^{\infty}(\Omega_\d)$ that approximates $\pi_\delta^1(u)$ strongly in $L^p(\Omega_\delta)$ and satisfies that $\phi_k$ is constant on $\Omega$ for every $k$. Then, $\Pcal_\d^{s_j}\phi_k \in N^{s_j, p, \delta}(\Omega)$ because of
\begin{align*}
D_\delta^{s_j} (\Pcal_\delta^{s_j} \phi_k)= \nabla (\Qcal_\delta^{s_j}\Pcal^{s_j}_\d \phi_k)=\nabla \phi_k=0 \quad \text{ on $\Omega$,}
\end{align*}
and, along with part $(i)$, 
\[
\lim_{k \to \infty} \lim_{j \to \infty} \norm{\Pcal_\delta^{s_j} \phi_k - \pi_\delta^1(u)}_{L^p(\Omega_\d)} = \lim_{k\to \infty}\norm{\varphi_k-\pi_\delta^1(u)}_{L^p(\Omega_\delta)}=0. 
\]
By extracting a suitable diagonal sequence, we obtain a sequence as in~\eqref{construct_hj}. \color{black}

\color{black}
Now, with $(h_j)_j$ and the convergences $u_j \to u$ and $\pi_\delta^{s_j}(u_j) \weakto w$ in $L^p(\Omega_\d)$  at hand, it follows that
\begin{align*}
\norm{u-\pi_\delta^1(u)}_{L^p(\Omega_\d)} &\leq \norm{u-w}_{L^p(\Omega_\d)} \leq \liminf_{j \to \infty} \norm{u_j-\pi_\delta^{s_j}(u_j)}_{L^p(\Omega_\d)} \\
&\leq \limsup_{j \to \infty} \norm{u_j-h_j}_{L^p(\Omega_\d)} = \norm{u-\pi_\delta^1(u)}_{L^p(\Omega_\d)}.
\end{align*}
As the inequalities in the previous lines turn to equalities, we infer $\pi_\delta^{s_j}(u_j) \to w = \pi_\delta^1(u)$ in $L^p(\Omega_\d)$, which finishes the proof of $(ii)$. \smallskip

\color{black}
\textit{Part $(iii)$.} By Corollary~\ref{cor:poincareN}, the sequence $(u_j)_j$ is bounded in $L^p(\Omega_\d)$. Using that the extension operator $\Ecal_\delta^s$ (see~Section~\ref{subsec:extension}) is uniformly bounded with respect to $s$ gives
\[
\sup_{j }\norm{\Ecal^{s_j}_\d u_j}_{H^{\bar{s},p}(\R^n)} \leq \sup_{j } \norm{\Ecal^{s_j}_\d u_j}_{H^{s_j,p}(\R^n)}<\infty,
\] 
with $\bar{s} \in (0,\inf_{j} s_j]$.  By the compact embedding of $H^{\bar{s}, p}(\R^n)$ into $L^p(\Omega_\delta)$, we can extract a subsequence (not relabeled) and find a $w\in L^p(\Omega_\d)$ such that $\Ecal^{s_j}_\d u_j \to w$ in $L^p(\Omega_\d)$. 
A distributional argument in analogy to~\cite[Lemma~9]{CKS23} allows us to deduce that $w|_\Omega\in W^{1,p}(\Omega)$, or equivalently, $w\in H^{1,p, \delta}(\Omega)$, and
 \begin{align}\label{eq287}
 D^{s_j}_\d u_j  = D^{s_j}_\d \Ecal^{s_j}_\d u_j \weakto \nabla w \qquad \text{ in $L^p(\Omega;\R^n)$ as $j\to \infty$. }
 \end{align}
Part $(ii)$ shows on the other hand that $\pi_\delta^{s_j}(\Ecal^{s_j}_\d u_j ) \to \pi_\delta^1(w)$ in $L^p(\Omega_\d)$ as $j\to \infty$. Hence,
 \begin{align}\label{eq654}
 u_j =  \Ecal_\delta^{s_j}u_j + (u_j -\Ecal_\delta^{s_j}u_j) = \Ecal_\delta^{s_j} u_j - \pi_\delta^{s_j}(\Ecal^{s_j}_\d u_j )  \to w-\pi_\delta^1(w) \quad  \text{in $L^p(\Omega)$ as $j\to \infty$;}
 \end{align}
note that the second equality is a consequence of  $u_j-\Ecal_\delta^{s_j}u_j\in N^{s_j, p, \delta}(\Omega)$, equation~\eqref{aux:pides}, and $u_j\in N^{s_j, p, \delta}(\Omega)^\perp$, which imply $\pi_\delta^{s_j}(\Ecal_\delta^{s_j}u_j) - \Ecal_{\delta}^{s_j}u_j + u_j = \pi_\delta^{s_j}(u_j)=0$.

 Finally, the statement follows from~\eqref{eq287} and~\eqref{eq654} with $u := w-\pi_\delta^1(w)\in N^{1,p, \delta}(\Omega)^\perp$, and the observation that $\pi_\delta^1(w)\in N^{1, p, \delta}(\Omega)$ is constant in $\Omega$. 
\end{proof}

We can now state and prove our localisation result in terms of variational convergence for $s\uparrow 1$.  Using the framework of $\Gamma$-convergence (see e.g.,~\cite{DalMaso, Braides}) guarantees the convergence of minimizers as a particular consequence.

\begin{theorem}[$\Gamma$-convergence to classical variational integral]\label{th:local}
Let $p \in (1,\infty)$, $F \in L^{p'}(\Omega_\d;\R^m)$ and $f:\Omega \times \Rmn \to \R_{\infty}$ be a Carath\'{e}odory integrand such that
\begin{align}\label{lowerbound2}
f(x,A) \geq c\,\bigl(\abs{A}^p-1\bigr) \quad \text{for a.e.~$x \in \Omega$ and all $A \in \Rmn$}
\end{align}
with a constant $c>0$. If $v \mapsto \int_\Omega f(x,\nabla v)\,dx$ is weakly lower semicontinuous on $W^{1,p}(\Omega;\R^m)$, then the  family of functionals $(\Fcal_\delta^s)_s$ with $\Fcal_\delta^s:L^p(\Omega_\d;\R^m) \to \R_{\infty}$ defined by
\[
\Fcal_\delta^s(u) = \begin{cases}
\displaystyle \int_{\Omega} f(x,D^s_\d u)\,dx-\int_{\Omega_\d}F \cdot u\,dx &\text{for $u \in {\Nzerom}^{\perp}$},\\
\infty &\text{else},
\end{cases}
\]
$\Gamma$-converge with respect to $L^p(\Omega_\d;\R^m)$-convergence as $s \to 1$ to $\Fcal_\delta^1:L^p(\Omega_\d;\R^m) \to \R_\infty$ given by 
\[
\Fcal^1_\delta(u) = \begin{cases}
\displaystyle \int_{\Omega}f(x,\nabla u)\,dx -\int_{\Omega} F\cdot u\,dx &\text{for $u \in N^{1,p, \delta}(\Omega;\R^m)^{\perp}$,}\\
\infty  &\text{else},
\end{cases}
\]
with $N^{1,p, \delta}(\Omega;\R^m)^\perp$ as in~\eqref{N1perp}. In addition, the family $(\Fcal_\delta^s)_s$ is equi-coercive in $L^p(\Omega_\d;\R^m)$.
\end{theorem}

\begin{proof} Let $(s_j)_j$ be a sequence in $(0,1)$ that converge to $1$ as $j\to \infty$.\smallskip

\textit{Step 1: Equi-coercivity.}  Let $(u_j)_j\subset L^p(\Omega_\delta)$ with $\sup_j \Fcal_{\delta}^{s_j}(u_j) <\infty$, in particular, $u_j\in N^{s_j, p, \delta}(\Omega)^\perp$ for each $j$. 
The lower bound~\eqref{lowerbound2} together with the nonlocal Poincar\'e inequality in Corollary~\ref{cor:poincareN} with a constant independent of $s$ shows that $(D_\delta^{s_j}u_j)_j$ is bounded in $L^p(\Omega;\R^{m\times n})$. Hence, the compactness result in Lemma~\ref{le:asymptotics}\,$(iii)$ is applicable and immediately yields a subsequence of $(u_j)_j$ that converges strongly in $L^p(\Omega;\R^m)$ to a function in $N^{1, p, \delta}(\Omega)^\perp$. 

\smallskip

\color{black}
\textit{Step 2: Liminf-inequality.} Let $(s_j)_j \subset (0,1)$ and $(u_j)_j\subset L^p(\Omega_\delta)$ be sequences such that $s_j \to 1$, $u_j \to u$ in $L^p(\Omega_\d;\R^m)$ as $j\to \infty$ and assume without loss of generality that $\sup_{j} \Fcal_\delta^{s_j}(u_j) < \infty$. 
Then, according to Lemma~\ref{le:asymptotics}\,$(iii)$ (cf.~also Step~1), $u \in N^{1,p, \delta}(\Omega)^\perp$ with $D^{s_j}_\d u_j \weakto \nabla u$ in $L^p(\Omega;\Rmn)$ as $j\to \infty$.  The desired liminf-inequality
\begin{align*}
\Fcal_\delta^{1}(u) \leq \liminf_{j\to \infty} \Fcal_{\delta}^{s_j}(u_j)
\end{align*}
is straightforward, if we exploit the weak lower semicontinuity of $v \mapsto \int_\Omega f(x,\nabla v)\,dx$ as in the proof of Theorem~\ref{th:wellposed}, but now with $\Qcal^{s_j}_\d$ varying with $j$.

 \smallskip

\textit{Step 3: Recovery sequence.} Let $u \in N^{1,p, \delta}(\Omega)^{\perp}$ with $\Fcal_\delta^1(u)<\infty$ and take $v \in W^{1,p}(\R^n;\R^m)$ with $v=u$ on $\Omega$. We define a sequence $(u_j)_j\subset L^p(\Omega_\delta)$ by setting
\[
u_j := \Pcal^{s_j}_\d v- \pi_\delta^{s_j}(\Pcal^{s_j}_\d v) \in N^{s_j, p, \delta}(\Omega;\R^m)^{\perp}.
\] 
By construction, it holds in view of~\eqref{translation_formula} that, for every $j$,
\begin{align}\label{construct_recovery}
D^{s_j}_\d u_j = D_\delta^{s_j} (\Pcal_\delta^{s_j} v) = \nabla v=\nabla u \quad \text{ on $\Omega$, }
\end{align}
and Lemma~\ref{le:asymptotics}\,$(i)$ and $(ii)$ imply
 \begin{align*}
 u_j \to v-\pi_\delta^1(v) = u \quad\text{ in $L^p(\Omega_\d;\R^m)$ as $j\to \infty$.}
 \end{align*} 
Observe that the identification of the limit function results from the fact that both $u$ and $v-\pi_\delta^1(v)$ lie in $N^{1,p, \delta}(\Omega)^{\perp}$ and they have the same gradient in $\Omega$. 
 
Altogether, we have shown that $u_j \to u$ in $L^p(\Omega_\d;\R^m)$ and 
\begin{align*}
\Fcal^{s_j}_\delta(u_j) = \int_{\Omega}f(x,D^{s_j}_\d u_j)\,dx - \int_{\Omega_\d} F \cdot u_j\,dx = \int_{\Omega}f(x,\nabla u)\,dx - \int_{\Omega_\d} F \cdot u_j\,dx \longrightarrow  \Fcal^1_\delta(u)
\end{align*}
as $j \to \infty$, which proves the stated $\Gamma$-convergence.
\end{proof}
\begin{remark} \label{rem:local1} We point out that the statement of Theorem~\ref{th:local} does not require any growth bound on $f$ from above. This is of particular relevance in settings with polyconvex integrands, which - motivated by applications in elasticity theory - are often chosen to be extended-valued. 
In terms of the proof, the waiver of any upper bound on $f$ is possible by the specific construction of the recovery sequence, whose nonlocal gradients are independent of $j$, see~\eqref{construct_recovery}. 
\end{remark}

Finally, we address what the previously shown convergence of the variational problems implies for the relation between local and nonlocal differential equations subject to natural and Neumann-type boundary conditions.

 Indeed, if the classical compatibility condition $\int_{\Omega} F\,dx = 0$ holds, then any minimizer $u \in L^p(\Omega_\d;\R^m)$ of $\Fcal_\delta^1$, when restricted to $\Omega$, also minimizes the functional
\[
v \mapsto \int_{\Omega}f(x,\nabla v)\,dx -\int_{\Omega} F\cdot v\,dx
\]
over the full space $W^{1,p}(\Omega;\R^m)$. In particular, if $f$ is continuously differentiable in its second argument with $f$ and $D_A f$ satisfying \eqref{eq:growth}, then the minimizer $u$ weakly satisfies the Euler-Lagrange system with natural boundary conditions
\begin{equation}\label{eq:natural}
\begin{cases}
-\Div (D_A f(\cdot,\nabla u)) = F & \text{in $\Omega$},\\
D_A f(\cdot,\nabla u) \cdot \nu = 0 & \text{on $\partial \Omega$}, 
\end{cases}
\end{equation}
where $\nu$ is an outward pointing unit normal to $\partial \Omega$. Therefore, Theorem~\ref{th:local} implies that the minimizers of $\Fcal_s^\d$ converge up to subsequence in $L^p(\Omega;\R^m)$ to a weak solution of \eqref{eq:natural} as $s \uparrow 1$.
\color{black}

\section*{Acknowledgements}
The authors would like to thank the Lorentz Center for their hospitality during the workshop \textit{``Nonlocality: Analysis, Numerics and Applications''}, which has inspired initial ideas for this work.

\end{document}